\documentclass[12pt]{article}
\usepackage[a4paper,margin=2cm]{geometry}
\usepackage[english]{babel}
\usepackage{latexsym,amsmath,enumerate,graphics,enumerate,amsthm,tikz,hyperref,float}
\usepackage[affil-it]{authblk}
\usepackage{enumerate,amsthm,dsfont,pstricks}
\usepackage{latexsym,amsmath,amssymb,amscd,wrapfig,graphicx}

\newtheorem{theorem}{Theorem}[section]
\newtheorem{lemma}[theorem]{Lemma}
\newtheorem{proposition}[theorem]{Proposition}
\newtheorem{corollary}[theorem]{Corollary}

\theoremstyle{definition}
\newtheorem{definition}[theorem]{Definition}

\theoremstyle{remark}
\newtheorem{remark}[theorem]{Remark}

\let\phi=\varphi
\def\epsilon{\varepsilon}

\def\R{\mathbb{R}}
\def\O{\mathbb{O}}

\def\eps{\varepsilon}

\def\0{\mathbf{0}}

\def\P{\mathcal{P}}
\def\CA{A_+^\circ}
\def\CB{B_+^\circ}
\def\ol{\overline}
\DeclareMathOperator{\diam}{diam}
\DeclareMathOperator{\Inv}{Inv}
\DeclareMathOperator{\Isom}{Isom}
\DeclareMathOperator{\Aut}{Aut}
\DeclareMathOperator{\Span}{Span}
\DeclareMathOperator{\Log}{Log}
\DeclareMathOperator{\Exp}{Exp}
\DeclareMathOperator{\esssup}{ess\; sup}

\newcommand{\comment}[1]{}
\newcommand{\norm}[1]{\left\Vert #1 \right\Vert}

\numberwithin{equation}{section}

\let\epsilon=\varepsilon

\makeatletter
\def\@maketitle{%
  \newpage
  \null
  \vskip 2em%
  \begin{center}%
  \let \footnote \thanks
    {\Large\bfseries \@title \par}%
    \vskip 1.5em%
    {\normalsize
      \lineskip .5em%
      \begin{tabular}[t]{c}%
        \@author
      \end{tabular}\par}%
    \vskip 1em%
    {\normalsize \@date}%
  \end{center}%
  \par
  \vskip 1.5em}
\makeatother

\begin{document}

\title{\sc \LARGE Hilbert and Thompson isometries on cones in JB-algebras}

\author{Bas Lemmens%
\thanks{Email: \texttt{B.Lemmens@kent.ac.uk}}}
\affil{School of Mathematics, Statistics \& Actuarial Science,
University  of Kent, Canterbury, Kent CT2 7NX, UK.}

\author{Mark Roelands%
\thanks{Email: \texttt{mark.roelands@gmail.com}}}
\affil{Unit for BMI, North-West University, Private Bag X6001-209, Potchefstroom 2520, South Africa}

\author{Marten Wortel%
\thanks{Email: \texttt{marten.wortel@gmail.com}}}
\affil{Unit for BMI, North-West University, Private Bag X6001-209, Potchefstroom 2520, South Africa}
\maketitle
\date{}

\begin{abstract}
Hilbert's and Thompson's metric spaces on the interior of cones in JB-algebras are important examples of symmetric Finsler spaces. In this paper we  characterize the Hilbert's metric isometries on the interiors of cones in JBW-algebras, and the Thompson's metric isometries on the interiors of cones in JB-algebras.  These characterizations  generalize work by Bosch\'e on the Hilbert's and Thompson's metric isometries on  symmetric cones, and work by  Hatori and Moln\'ar on  the Thompson's metric isometries  on the cone of positive selfadjoint elements in a unital $C^*$-algebra.  To obtain the results  we develop  a variety of new  geometric and Jordan algebraic techniques.
\end{abstract}




\section{Introduction}

On the interior $A_+^\circ$ of the cone in an order unit space $A$ there exist two important metrics: Hilbert's metric and Thompson's metric. Hilbert's metric goes back to  Hilbert \cite{Hil}, who defined a metric $\delta_H$ on an open bounded convex set $\Omega$ in a finite dimensional real vector space $V$ by
\[
\delta_H(a,b):=\log \left(\frac{\|a'-b\|}{\|a'-a\|}\frac{\|b'-a\|}{\|b'-b\|}\right),
\]
 where $a'$ and $b'$ are the points of intersection of the line through $a$ and $b$ and $\partial \Omega$ such that $a$ is between $a'$ and $b$, and $b$ is between $b'$ and $a$.
The Hilbert's metric spaces $(\Omega,\delta_H)$  are Finsler manifolds that  generalize Klein's model of the real hyperbolic space. They  play a role in the solution of Hilbert's Fourth problem \cite{AP}, and possess features of nonpositive curvature \cite{Ben,KN}. In recent years there has been increased interest in the geometry of Hilbert's metric spaces, see \cite{Handb} for an overview. In this paper we shall work with a slightly more general version of Hilbert's metric, which is a metric between pairs of the rays in the interior of the cone. It is defined  in terms of the partial ordering of the cone and  was introduced by Birkhoff \cite{Bi}. It has found numerous applications in the spectral theory of linear and nonlinear operators, ergodic theory, and fractal analysis,  see \cite{LNBook,LNSurv,Liv,Metz,NMem,Nu,Sa} and the references therein.

Thompson's metric was introduced by Thompson in \cite{Thom}, and is also  a useful tool in the spectral theory of operators on cones. If the order unit space is complete, the resulting Thompson's metric space is a prime example of a Banach-Finsler manifold.  Moreover, if the order unit space is a JB-algebra (which is a simultaneous generalization of both a Euclidean Jordan algebra as well as the selfadjoint elements of a $C^*$-algebra), then the Banach-Finsler manifold is symmetric and possesses certain features of nonpositive curvature \cite{ACS,CPR1,CPR2,LL1,LL2,Lim,Neeb,Nu,Up}.  This is one of the main reasons why Thompson's metric is of interest in the study of the geometry of spaces of positive operators.

It appears that  understanding the isometries of Hilbert's and Thompson's  metrics on the interiors of cones in order unit spaces is closely linked with the theory of JB-algebras. Evidence for this link was provided by Walsh \cite{Wa}, who showed, among other things, that for finite dimensional order unit spaces $A$,  the Hilbert's metric isometry group on $A_+^\circ$ is not equal to the group of projectivities of $A_+^\circ$ if and only if $A$ is a Euclidean Jordan algebra whose cone is not Lorentzian (\cite[Corollary~1.4]{Wa}). Moreover, in that case,  the group of projectivities  has index 2 in the isometry group, and the additional isometries are obtained by adjoining  the map induced by $ a\in A_+^\circ\mapsto a^{-1}\in A_+^\circ$.  At present it is unknown if this  result has an infinite dimensional extension.

The main objective of this paper is to characterize the Hilbert's metric isometries on the interiors of cones in JBW-algebras (a subclass of JB-algebras that includes both the selfadjoint elements of von Neumann algebras as well as Euclidean Jordan algebras), and the Thompson's metric isometries on the interiors of cones in JB-algebras.  Unfortunately our methods do not yield a characterization of the Hilbert's metric isometries for general JB-algebras, as we require the existence of sufficiently many projections.

Our results generalize and complement a number of earlier works. In 2009 Moln\'ar \cite{Mol1} described the Thompson's and Hilbert's metric isometries on the selfadjoint operators on a Hilbert space of dimension at least three, using geometric means to show that these isometries preserve commutativity. A different approach was taken in 2010 by Hatori and Moln\'ar \cite{HM} where they characterized Thompson's metric isometries on the positive cone of a $C$*-algebra by showing that they induce a linear norm isometry on the selfadjoint elements of the $C$*-algebra. Similarly, Bosch\'e \cite{Bo} described Thompson's and Hilbert's metric isometries on a symmetric cone in 2012 by showing that they induce norm isometries on the whole Euclidean Jordan algebra: a Thompson's metric isometry yields a linear JB-norm isometry, and a Hilbert's metric isometry yields a linear variation norm isometry. On JB-algebras, linear variation norm isometries are exactly linear maps preserving the maximal deviation, the quantum analogue of the maximal standard deviation, see \cite{Mol2,MB, Ham}. These were charactized on the selfadjoint elements of von Neumann algebras without a type $I_2$ summand by Molnar \cite{Mol2} in 2010, and in 2012 this result was extended to JBW-algebras without a type $I_2$ summand by Hamhalter \cite{Ham}.

Our approach is to show that Thompson's and Hilbert's metric isometries on the positive cone of a JB-algebras induce linear norm isometries on the whole JB-algebra: the Thompson's metric isometries yield norm isometries, whereas the Hilbert's metric isometries induce variation norm isometries, see Theorem~\ref{t:inducing_isometry}. This extends the approach in \cite{Bo} and \cite{HM}. By using a characterization of linear norm isometries of JB-algebras due to Isidro and Rodr\'iguez-Palacios \cite{IRP} we then characterize the Thompson's metric isometries of JB-algebras, generalizing results of \cite{Bo} and \cite{HM}. 
As for Hilbert's metric, we restrict to JBW-algebras. If there is no type $I_2$ summand, Hamhalter's characterization of the linear variation norm isometries mentioned above yields the desired description of the Hilbert's metric isometry. But in general this result can not be used, so we exploit the fact that in our case the variation norm isometry is induced by a Hilbert's metric isometry to obtain the desired characterization.

 This characterization also complements our earlier work \cite{LRW1}, in which we  considered the order unit space $C(K)$ consisting of all continuous functions on a compact Hausdorff space $K$. In the same paper we showed that the group of Hilbert's metric isometries is equal to the group of projectivities if the Hilbert's metric is uniquely geodesic.  Other works on Hilbert's metric isometries and Thompson's metric isometries on finite dimensional cones include \cite{dlH, LW, MT, Sp}.

\bigskip

The structure of the paper is as follows. 

Section~\ref{sec:prelims} is our preliminary section. We first introduce Thompson's and Hilbert's metrics and JB(W)-algebras. We then investigate some properties that will prove to be very useful in characterizing the isometries for both metrics. In particular, we characterize when there exist unique geodesics for Thompson's and Hilbert's metric  between two elements of a JB-algebra, and we study the interplay between geometric means and the isometries for both metrics. These findings also generalize earlier work done on Euclidean Jordan algebras and $C^*$-algebras, and result in the crucial Theorem~\ref{t:inducing_isometry} mentioned above.

In Section~\ref{sec:thomp} we characterize the isometries for Thompson's metric, and we exploit this result to describe the corresponding isometry group of a direct product of simple JB-algebras in terms of the automorphism groups of the components.

Finally, we consider Hilbert's metric isometries in Section~\ref{sec:hilbert}. Since the extreme points of the unit ball in the quotient coincide with the equivalence classes of nontrivial projections, every Hilbert's metric isometry induces a bijection on the projections. At this point we restrict to JBW-algebras as they contain a lot of projections in contrast to JB-algebras. By using geometric properties of Hilbert's metric as well as operator algebraic methods, we obtain that the above bijection on the projections is actually a projection orthoisomorphism: two projections are orthogonal if and only if their images are orthogonal. Dye's classical theorem \cite{Dye} shows that every projection orthoisomorphism between von Neumann algebras without a type $I_2$ summand extends to a Jordan isomorphism on the whole algebra. This was extended by Bunce and Wright \cite{BW} to JBW-algebras, and we use this result to extend our projection orthoisomorphism defined outside the type $I_2$ summand to a Jordan isomorphism. It remains to take care of the type $I_2$ summand, which we are able to do using a characterization of type $I_2$ JBW-algebras due to Stacey \cite{S} and the explicit fact that our projection orthoisomorphism comes from a linear map on the quotient. Thus we are able to extends the whole projection orthoisomorphism to a Jordan isomorphism, which then easily yields the main result of our paper, Theorem~\ref{t:hilbertisoms}, which we repeat below for the reader's convenience. The set $\ol{M}_+^\circ$ denotes the set of rays in $M_+^\circ $, and $U_b$ denotes the quadratic representation of $b$. 

\begin{theorem}\label{t:introduction}
If $M$ and $N$ are JBW-algebras, then $f \colon \ol{M}_+^\circ \to \ol{N}_+^\circ$ is a bijective Hilbert's metric isometry if and only if
\begin{equation*}
 f(\overline{a}) = \ol{U_b J(a^\epsilon)} \mbox{\quad for all }\ol{a}\in \ol{M}_+^\circ,
 \end{equation*}
where $\epsilon\in\{-1,1\}$, $b\in N_+^\circ$, and $J\colon M\to N$ is a Jordan isomorphism. In this case $b\in f(\ol{e})^{\frac{1}{2}}$.
\end{theorem}

We claim that this result extends Molnar's theorem (\cite[Theorem~2]{Mol1}), reformulated below using our notation.

\begin{theorem}[Molnar]
Let $H$ be a complex Hilbert space with $\dim(H) \geq 3$ and let $f \colon B(H)_+^\circ \to B(H)_+^\circ$ be a bijective Hilbert's metric isometry. Then there is an invertible bounded linear or conjugate linear operator $z \colon H \to H$ and an $\eps \in \{\pm 1 \}$ such that
$$ f(\overline{a}) =  \overline{za^{\eps}z^*}. $$
\end{theorem}

Indeed, \cite[Theorem~2.2]{IRP} states that all Jordan isomorphisms $J$ of $B(H)$ are of the form $Ja = uau^*$, where $u$ is a unitary or anti-unitary (i.e., conjugate linear unitary) operator. Hence
$$ U_b J(a^\eps) = bu a^\eps u^* b = (bu) a^\eps (bu)^*. $$
It remains to show that any invertible (conjugate) linear operator $z \in B(H)$ can be written as $bu$, with a positive $b$ and (anti-)unitary $u$. For linear operators this is just the polar decomposition, and by considering a conjugate linear operator to be a linear operator from $H$ to its conjugate Hilbert space, we obtain the same decomposition for conjugate linear operators.

In view of \cite[Corollary~1.4]{Wa} mentioned above we make the following contribution in Proposition~\ref{isom group}, where we show that the isometry group for Hilbert's metric on JBW-algebras is not equal to the group of projectivities if and only if the cone is not a Lorentz cone.

\section{Preliminaries}\label{sec:prelims}
In this section we collect some basic definitions  and recall several useful facts concerning Hilbert's and Thompson's metrics and cones in JB-algebras.

\subsection{Order unit spaces}
Let $A$ be a partially ordered real vector space with cone $A_+$. So, $A_+$ is convex, $\lambda A_+\subseteq A_+$ for all $\lambda\ge 0$, $A_+\cap -A_+=\{0\}$, and the  partial ordering $\leq$ on $A$ is given by $a\leq b$ if $b-a\in A_+$. Suppose that there exists an {\em order unit} $u\in A_+$, i.e., for each $a\in A$ there exists $\lambda>0$ such that $-\lambda u\leq a\leq \lambda u$. Furthermore assume that $A$ is {\em Archimedean}, that is to say, if $na\leq u$ for all $n=1,2,\ldots$, then $a\leq 0$.  In that case $A$ can be equipped with the {\em order unit norm},
\[
\|a\|_u:=\inf\{\lambda>0 \colon -\lambda u\leq a\leq \lambda u\},
\]
and $(A,\|\cdot\|_u)$ is called an {\em order unit space}, see \cite{HO}.
It is not hard to show, see for example \cite{LRW1}, that $A_+$ has nonempty interior $A_+^\circ$ in $(A,\|\cdot\|_u)$ and $A_+^\circ=\{a\in A\colon \mbox{$a$ is an order unit of $A$}\}$.

On $A_+^\circ$  Hilbert's metric and Thompson's metric are defined as follows. For $a,b\in A_+^\circ$ let
\[
M(a/b):=\inf\{\beta>0\colon a\leq \beta b\}.
\]
Note that as $b\in A_+^\circ$ is an order unit, $M(a/b)<\infty$. On $A_+^\circ$, {\em Hilbert's metric} is given by
\begin{equation}\label{d_H}
d_H(a,b) = \log M(a/b)M(b/a),
\end{equation}
and {\em Thompson's metric} is defined by
\begin{equation}\label{d_T}
d_T(a,b) = \log \max\{M(a/b),M(b/a)\}.
\end{equation}
It is well known (cf.  \cite{LNBook,NMem}) that $d_T$ is a metric on $A_+^\circ$, but $d_H$ is not, as $d_H(\lambda a,\mu b)=d_H(a,b)$ for all $\lambda,\mu>0$ and $a,b\in A_+^\circ$. However, $d_H(a,b)=0$ for $a,b\in A_+^\circ$ if and only if $a=\lambda b$ for some $\lambda>0$, so  that $d_H$ is a metric on the set of  rays in $A_+^\circ$, which we shall denote by $\overline{A}_+^\circ$. Elements of $\overline{A}_+^\circ$ will be denoted by $\ol{a}$, and if $\Omega \subseteq A_+^\circ$ the set of rays through $\Omega$ will be denoted by $\overline{\Omega}$.

\subsection{JB-algebras}
A \emph{Jordan algebra} $(A, \circ)$ is a commutative, not necessarily associative algebra such that
\[
a \circ (b \circ a^2) = (a \circ b) \circ a^2 \mbox{\quad  for all }a,b \in A.
\]
A \emph{JB-algebra} $A$ is a normed, complete real Jordan algebra satisfying,
\begin{align*}
\norm{a \circ b} &\leq \norm{a}\norm{b}, \\
\norm{a^2} &= \norm{a}^2, \\
\norm{a^2} &\leq \norm{a^2 + b^2}
\end{align*}
for all $a,b \in A$. An important example of a JB-algebra is the set of selfadjoint elements of a $C^*$-algebra $A$, equipped with the Jordan product $a \circ b := (ab + ba)/2$. By the Gelfand-Naimark theorem, this JB-algebra is a norm closed Jordan subalgebra of the selfadjoint bounded operators on a Hilbert space; such an algebra is called a \emph{JC-algebra}. By \cite[Corollary 3.1.7]{HO}, Euclidean Jordan algebras are another example of JB-algebras. We can think of JB-algebras as a simultaneous generalization of both the selfadjoint elements of $C^*$-algebras as well as Euclidean Jordan algebras.

Throughout the paper, we will assume that all JB-algebras are unital with  unit  $e$.

The set of invertible elements of $A$ is denoted by $\Inv(A)$. The \emph{spectrum} of $a \in A$, $\sigma(a)$, is defined to be the set of $\lambda \in \R$ such that $a - \lambda e$ is not invertible in JB$(a,e)$, the JB-algebra generated by $a$ and $e$ (\cite[3.2.3]{HO}). There is a continuous functional calculus: JB$(a,e) \cong C(\sigma(a))$. Both the spectrum and the functional calculus coincide with the usual notions in both Euclidean Jordan algebras as well as JC-algebras.

The elements $a,b \in A$ are said to \emph{operator commute} if $a \circ (b \circ c) = b \circ (a \circ c)$ for all $c \in A$. In a JC-algebra, two elements operator commute if and only if they commute in the $C^*$-multiplication (\cite[Proposition~1.49]{AS}). In the sequel we shall write the Jordan product of two operator commuting elements $a,b\in A$ as $ab$ instead of $a\circ b$. The \emph{center} of $A$ consists of all elements that operator commute with all elements of $A$, and it is an associative JB-subalgebra of $A$. Every associative JB-algebra is isomorphic to $C(K)$ for some compact Hausdorff space $K$ (\cite[Theorem 3.2.2]{HO}).

The cone of elements with nonnegative spectrum is denoted by $A_+$, and equals the set of squares by the functional calculus, and its interior $A_+^\circ$ consists of all elements with strictly positive spectrum, or equivalently, all  elements in $A_+\cap\mathrm{Inv}(A)$. This cone turns $A$ into an order unit space with order unit $e$, i.e.,
\[ \norm{a} = \inf \{ \lambda > 0: -\lambda e \leq a \leq \lambda e \}. \]
Note that the JB-norm is not the same as the usual norm in a Euclidean Jordan algebra.

The \emph{Jordan triple product} $\{ \cdot, \cdot, \cdot \}$ is defined as
\[ \{a,b,c \} := (a \circ b) \circ c + (c \circ b) \circ a - (a \circ c) \circ b, \]
for $a,b,c \in A$. In a JC-algebra one easily verifies that $\{a,b,c\} = (abc+cba)/2$. For $a \in A$, the linear map $U_a \colon A \to A$ defined by $U_a b := \{a,b,a\}$ will play an important role and  is called the \emph{quadratic representation} of $a$.

By the Shirshov-Cohn theorem for JB-algebras \cite[Theorem 7.2.5]{HO}, the unital JB-algebra generated by two elements is a JC-algebra, which shows all but the fifth of the following identities for JB-algebras, since $U_a b = aba$  in JC-algebras.  (For the rest of the paper, the operator-algebraic reader is encouraged to think of this equality whenever the quadratic representation appears.)
\begin{equation}
\begin{aligned}
(U_a b)^2 &= U_a U_b a^2 \quad &\forall a,b \in A. \\
U_a b &\in A_+ \quad &\forall a \in A, \forall b \in A_+. \\
U_a^{-1} &= U_{a^{-1}} \quad &\forall a \in \Inv(A). \\
(U_a b)^{-1} &= U_{a^{-1}} b^{-1} \quad &\forall a,b \in \Inv(A). \\
U_{U_a b} &= U_a U_b U_a \quad &\forall a,b \in A. \\
U_a e &= a^2 \quad &\forall a \in A. \\
U_{a^\lambda} a^\mu &= a^{2\lambda + \mu} \quad &\forall a \in A, \forall \lambda, \mu \in \R.
\end{aligned}
\label{e:U_facts}
\end{equation}
A proof of the fifth identity can be found in \cite[2.4.18]{HO}, as well as  proofs of the other identities.

A JB-algebra $A$ induces an algebra structure on $\ol{A}_+^\circ$ by $\ol{a}\circ\ol{b}:=\ol{a\circ b}$, which is well-defined. We can also define $\ol{a}^{\alpha} := \ol{a^\alpha}$ for $\alpha\in\mathbb{R}$.  For $a\in \mathrm{inv}(A)$, the quadratic representation $U_a$ is an order isomorphism, and induces a well defined map $U_{\ol{a}}$ on $\ol{A}_+^\circ$ by
\[
U_{\ol{a}}(\ol{b}) := \ol{U_a(b)}\mbox{\quad for all }\ol{b}\in \ol{A}_+^\circ.
\]
When studying Hilbert's metric on  $\ol{A}_+^\circ$ in JB-algebras, the  \emph{variation seminorm}
$\norm{\cdot}_v$ on $A$ given by,
\[
\norm{a}_v := \diam \sigma(a)=\max\sigma(a)-\min\sigma(a),
\]
will play an important role. The kernel of this seminorm is the span of $e$, and on the quotient space $[A]:=A / \Span(e)$ it is a norm. To see this we show that if
$\|\cdot\|_q$ is the quotient norm of $2\|\cdot\|$ on $[A]$, then $\|[a]\|_q =\|[a]\|_v$ for all $[a]\in [A]$.   Indeed, for $[a]\in [A]$, using $\inf_{\lambda\in\R}\max\{t-\lambda,s+\lambda\}=(t+s)/2$, we have that
\begin{eqnarray*}
\|[a]\|_q & := & 2\inf_{\mu\in\mathbb{R}}\|a -\mu e\|\\
 & = & 2\inf_{\mu\in\mathbb{R}} \max_{\lambda \in \sigma(a)}|\lambda -\mu|\\
  & = & 2\inf_{\mu\in\mathbb{R}} \max\big\{{\textstyle\max_{\lambda \in \sigma(a)}}(\lambda -\mu), {\textstyle\max_{\lambda \in \sigma(a)}}(-\lambda +\mu)\big\}\\
  & = &  \max\sigma(a) + \max-\sigma(a)=\max\sigma(a)-\min\sigma(a)\\
& = & \|[a]\|_v.
\end{eqnarray*}
Note that the map $\Log\colon A_+^\circ\to A$ given by $a\mapsto \log(a)$ is a bijection,  whose inverse $\Exp$ is given by $a\mapsto \exp(a)$. Furthermore, as  $\log(\lambda a ) =\log(a)+ \log(\lambda) e$ for all $a\in A_+^\circ$ and $\lambda>0$, the map $\Log$ induces a bijection from $\ol{A}_+^\circ$ onto $[A]$ given by $\log\ol{a} = [\log a]$. Its inverse $\mathrm{Exp}\colon [A]\to \ol{A}_+^\circ$ is given by $\exp([a]) = \ol{\exp(a)}$ for $[a]\in [A]$.

A \emph{JBW-algebra} is the Jordan analogue of a von Neumann algebra: it is a JB-algebra which is monotone complete and has a separating set of normal states, or equivalently, a JB-algebra that is a dual space. In JBW-algebras the spectral theorem holds, which implies in particular that the linear span of projections is norm dense. If $p$ is a projection, then the complement $e-p$ will be denoted by $p^\perp$. Every JBW-algebra decomposes into a direct sum of a type I, II, and III JBW-algebras. A JBW-algebra with trivial center is called a \emph{factor}. Every Euclidean Jordan algebra is a JBW-algebra, and a Euclidean Jordan algebra is simple if and only if it is a factor.

\subsection{Order isomorphisms}
An important result we use is \cite[Theorem~1.4]{IRP}, which we state here for the convenience of the reader. A {\em symmetry} is an element $s$ satisfying $s^2 = e$. Note that $s$ is a symmetry if and only if $p := (s+e)/2$ is a projection, and $s = p - p^\perp$.

\begin{theorem}[Isidro, Rodr{\'{\i}}guez-Palacios]\label{t:jbisoms}
 The bijective linear isometries from $A$ onto $B$ are the mappings of the form $a \mapsto s Ja$, where $s$ is a
central symmetry in $B$ and $J \colon A \to B$ a Jordan isomorphism.
\end{theorem}

This theorem uses the fact that a bijective unital linear isometry between JB-algebras is a Jordan isomorphism, which is \cite[Theorem~4]{WY}. We use this simpler statement in the following corollary.

\begin{corollary}\label{orderisoms}
Let $A$ and $B$ be order unit spaces, and $T \colon A \to B$ be a unital linear bijection. Then $T$ is an isometry if and only if $T$ is an order isomorphism. Moreover, if $A$ and $B$ are JB-algebras, then these statements are equivalent to $T$ being a Jordan isomorphism.
\end{corollary}
\begin{proof}
Suppose $T$ is an isometry, and let $a \in A_+$, $\norm{a} \leq 1$. Then $\norm{e-a} \leq 1$, and so $\norm{e - Ta} \leq 1$, showing that $Ta$ is positive. So $T$ is a positive map, and by the same argument $T^{-1}$ is a positive map. (This argument is taken from the first part of \cite[Theorem 4]{WY}.)

Conversely, if $T$ is an order isomorphism, then $-\lambda e \leq a \leq \lambda e$ if and only if $-\lambda e \leq Ta \leq \lambda e$, and so $T$ is an isometry.

Now suppose that $A$ and $B$ are JB-algebras. If $T$ is an isometry, then $T$ is a Jordan isomorphism by \cite[Theorem 4]{WY}. Conversely, if $T$ is a Jordan isomorphism, then $T$ preserves the spectrum, and then also the norm since $\norm{a} = \max |\sigma(a)|$.
\end{proof}

This corollary will be used to show the following proposition. For Euclidean Jordan algebras this proposition has been proved in \cite[Theorem III.5.1]{FK}.

\begin{proposition}\label{p:order_isomorphism}
A map $T \colon A \to B$ is an order isomorphism if and only if $T$ is of the form $T = U_b J$, where $b \in \CB$ and $J$ is a Jordan isomorphism. Moreover, this decomposition is unique and $b=(Te)^{\frac{1}{2}}$.
\end{proposition}
\begin{proof}
If $T$ is of the above form, then $T$ is an order isomorphism as a composition of two order isomorphisms. Conversely, if $T$ is an order isomorphism, then $T = U_{(Te)^{\frac{1}{2}}} U_{(Te)^{-\frac{1}{2}}} T$, and by the above corollary $U_{(Te)^{-\frac{1}{2}}} T$ is a Jordan isomorphism.

For the uniqueness, if $T = U_b J$, then $Te = U_b Je = U_b e = b^2$ which forces $b = (Te)^{\frac{1}{2}}$. This implies that $J = U_{(Te)^{-\frac{1}{2}}} T$, so $J$ is also unique.
\end{proof}

\subsection{Hilbert's and Thompson's metrics on cones in JB-algebras}

Suppose $A$ is a JB-algebra. For $c \in \CA$, the map $U_c$ is an order isomorphism of $A$, and hence it  preserves $M(a/b)$. Thus, $U_c$ is an isometry under $d_H$ and $d_T$. This can be used to derive the following expressions for $d_H$ and $d_T$ on cones in JB-algebras.
\begin{proposition}\label{p:hilbert_thompson_metrics}
If $A$ is a JB-algebra and $a,b\in \CA$, then
\[
d_H(\ol{a},\ol{b}) = \norm{\log U_{b^{-\frac{1}{2}}}a}_v \mbox{\quad and\quad }
 d_T(a,b) =  \norm{ \log U_{b^{-\frac{1}{2}}}a} .
\]
\end{proposition}
\begin{proof}
Since $U_c$ is an order isomorphism of $A$ for $c \in \CA$,
\[
  \inf \{ \lambda > 0\colon a\leq \lambda b \}
=  \inf \{ \lambda > 0\colon  U_{b^{-\frac{1}{2}}}a\leq \lambda e \}
= \max \sigma(U_{b^{-\frac{1}{2}}}a),
\]
and hence $\log M(a/b) = \log \max \sigma(U_{b^{-\frac{1}{2}}}a) =\max \sigma( \log  U_{b^{-\frac{1}{2}}}a )$.

Similarly,
\[
  \inf \{ \lambda > 0\colon b\leq \lambda a  \}
 =  (\sup \{ \mu > 0 \colon  \mu b \leq a \})^{-1}
 =(\sup \{ \mu > 0 \colon \mu e\leq  U_{b^{-\frac{1}{2}}}a \})^{-1}
= (\min \sigma(U_{b^{-\frac{1}{2}}}a))^{-1}
\]
gives $\log M(b/a) = \log (\min \sigma(U_{b^{-\frac{1}{2}}}a))^{-1} = - \min \sigma( \log U_{b^{-\frac{1}{2}}}a)$.

The formula for $d_H$ follows immediately. As $\norm{c} = \max\{ \max \sigma(c), - \min \sigma(c)\}$ for $c \in A$,
the identity for $d_T$ holds.
\end{proof}
Also note that the inverse map on $\CA$ satisfies $M(b^{-1}/a^{-1}) = M(a/b)$, so this is an isometry for both metrics as well. Indeed, using (\ref{e:U_facts}) we see that
\begin{eqnarray*}
M(b^{-1}/a^{-1}) & = & \inf\{ \lambda >0 \colon b^{-1}\leq \lambda a^{-1}\}\\
 & = & \inf\{ \lambda >0 \colon e\leq \lambda U_{b^{\frac{1}{2}}}a^{-1}\}\\
 & = & \inf\{ \lambda >0 \colon e\leq \lambda (U_{b^{-\frac{1}{2}}}a)^{-1}\}\\
 & = & \inf\{ \lambda >0 \colon  U_{(U_{b^{-\frac{1}{2}}}a)^{\frac{1}{2}}}e\leq \lambda e\}\\
 & = & \inf\{ \lambda >0 \colon  U_{b^{-\frac{1}{2}}}a\leq \lambda e\}\\
 &  = &   M(a/b).
\end{eqnarray*}

Given a JB-algebra $A$ we follow Bosch\'e \cite[Proposition 2.6]{Bo} and Hatori and Moln\'ar \cite[Theorem 9]{HM}, and introduce for $n \geq 1$ metrics  on $[A]$ and $A$, respectively, by
\[
 d_n^H([a],[b]) := n d_H(\exp([a]/n), \exp([b]/n)) \mbox{\quad and\quad }
 d_n^T(a,b) := n d_T(\exp(a/n), \exp(b/n))
  \]
 for all $a,b\in A$.
Note that $d_n^H$ is well defined, because if $a_1,a_2\in [a]$, then $\exp(a_1/n) =\lambda\exp(a_2/n)$ for some $\lambda>0$.
\begin{proposition}\label{p:d_n_converges_to_norm}
If $A$ is a JB-algebra and  $a,b \in A$, then
\[ \lim_{n\to\infty} d_n^H([a],[b]) = \norm{[a]-[b]}_v \mbox{\quad and\quad }
\lim_{n\to\infty}d_n^T(a,b) = \norm{a-b}. \]
\end{proposition}
\begin{proof}
We start with some preparations. The JB-algebra generated by $a$, $b$ and $e$ is special, so we can think of $U_{\exp(b/n)^{-\frac{1}{2}}} \exp(a/n)$ as $\exp(-b/2n) \exp(a/n) \exp(-b/2n)$ for some $C^*$-algebra multiplication. Writing out the exponentials in power series yields
\[
U_{\exp(b/n)^{-\frac{1}{2}}} \exp(a/n) = e + (a-b)/n + o(1/n).
\]
Furthermore, using the power series representation,
\[ \log(e+c) = \sum_{k=1}^\infty \frac{(-1)^{k+1} c^k}{k}, \]
which is valid for $\norm{c} < 1$, we obtain for sufficiently large $n$ that
\[ \log\left(U_{\exp(b/n)^{-\frac{1}{2}}} \exp(a/n)\right) = (a-b)/n + o(1/n).\]

So, for all sufficiently large $n$ we have by Proposition \ref{p:hilbert_thompson_metrics} that
\begin{eqnarray*}
\left| d^H_n([a],[b]) -\|[a]-[b]\|_v\right|& =& \left| n d_H(\exp(a/n), \exp(b/n)) -\|a-b\|_v\right|\\
 & = & \left| n \norm{ \log\left( U_{\exp(b/n)^{-\frac{1}{2}}} \exp(a/n)\right)}_v-\|a-b\|_v\right|\\
 & = & \left| \| a-b +no(1/n)\|_v -\|a-b\|_v\right|\\
 & \leq &  n \|o(1/n)\|_v\\
 & \leq & 2n\|o(1/n)\|.
\end{eqnarray*}
As the right hand side converges to 0 for $n\to\infty$, the first limit holds. The second limit can be derived in the same way.
\end{proof}

We will also need  some basic facts concerning the unique geodesics for $d_T$ and $d_H$. Recall that for a metric space $(M,d)$ a map $\gamma\colon I \to M$, where $I$ is a possibly unbounded interval in $\mathbb{R}$, is a {\em geodesic path} if there is a $k\geq 0$ such that $d(\gamma(s),\gamma(t)) = k|s-t|$ for all $s,t \in I$. The image of a geodesic path is called a {\em geodesic}. The following result generalizes \cite[Theorems 5.1 and 6.2]{LR}.
\begin{theorem}\label{Thom_unique}
If $A$ is a JB-algebra and $a,b \in A_+^\circ$ are linearly independent, then there exists a unique Thompson geodesic between $a$ and $b$ if and only
if $\sigma(U_{a^{-\frac{1}{2}}}b) = \{ \beta^{-1}, \beta \}$ for some $\beta > 1$.
\end{theorem}
\begin{proof}
As the map $U_{a^{-\frac{1}{2}}}$ is  a Thompson's metric isometry,  we may assume without loss of generality that $a = e$. First suppose that $\sigma(b) = \{\beta^{-1}, \beta \}$ for some $\beta > 1$, then $b = \beta^{-1}p + \beta p^\perp$ and the line through $b$ and $e$ intersects $\partial A_+$ in $\lambda p$ and $\mu p^\perp$ for some $\lambda,\mu>0$. We wish to apply \cite[Theorem~4.3]{LR}.

Consider the Peirce decomposition $ A = A_1 \oplus A_{1/2} \oplus A_0$ (cf. \cite[2.6.2]{HO}) with respect to $p$. We denote the projection onto $A_i$ by $P_i$, for $i=1, 1/2, 0$. Then $P_1 = U_p$ and $P_0 = U_{p^\perp}$. From \cite[Proposition~1.3.8]{AS} we know that if $a \in A_+$, then $U_p a = a$ if and only if $U_{p^\perp} a = 0$. Using this result we now prove the following claim.

{\em Claim.} Let $v \in A$. If $\alpha,\delta > 0$ and $p \in A$ is a projection such that $\alpha p + tv \in A_+$ for all $|t| < \delta$, then $v \in A_1$.

To show the claim, note that  $0 \leq U_{p^\perp}(\alpha p+tv) = t U_{p^\perp} v$ for all $|t| < \delta$, so that $U_{p^\perp} v = 0$, and consequently $U_{p^\perp}(\alpha p+tv) = t U_{p^\perp} v = 0$ for all $|t| <
\delta$. Let $0 <|t| < \delta$ be arbitrary. It follows that $\alpha p + tv = U_p( \alpha p + tv) = \alpha p + t U_p v$ and so $v = U_p v = P_1
v$, i.e., $v \in A_1$.

By applying the claim to $\lambda p$ as well as $\mu p^\perp$, it follows that if $v \in A$ is such that $\lambda p + tv \in A_+$ and $\mu p^\perp + tv \in A_+$ for all $|t| < \delta$, then $v \in A_1 \cap A_0 = \{0\}$. Hence, by \cite[Theorem 4.3]{LR}, there is a unique geodesic between $b$ and $e$.

Conversely, suppose that there is a unique geodesic between $b$ and $e$. Then this is also a unique geodesic in JB$(b,e) \cong C(\sigma(b))$. For $f,g \in C(\sigma(b))$ we have by Proposition \ref{p:hilbert_thompson_metrics} that
\[
 d_T(f,g) = \norm{ \log U_{g^{-\frac{1}{2}}} f} = \sup_{k \in \sigma(b)} \left|\log \frac{f(k)}{g(k)} \right| = \sup_{k \in \sigma(b)} |\log f(k) - \log g(k)| = \norm{\log f - \log g}.
 \]
So, the pointwise logarithm is an isometry from
$ (C(\sigma(b)^\circ_+), d_T)$ onto  $(C(\sigma(b)), \norm{\cdot}_\infty)$,
which sends $e$ to the zero function and $b$ to the function $k \mapsto \log k$.

Note that for $f\in C(\sigma(b))$ the images of both $t\mapsto(t\|f\|\wedge|f|)\mathrm{sgn}f$ and $t\mapsto t f$ are geodesics connecting $0$ and $f$, which are different if and only if there is a point $k\in \sigma(b)$ such that $|f(k)|\not=\|f\|$. Hence $k \mapsto |\log(k)|$ is constant. So, if $\alpha,\beta \in \sigma(b)$, then  $|\log \beta| = |\log \alpha|$, and hence $\alpha=\beta$ or $\alpha =\beta^{-1}$. This shows that $\sigma(b) \subseteq \{ \beta^{-1}, \beta \}$, and since $b$ and $e$ are linearly independent we must have equality.
\end{proof}

From Theorem \ref{Thom_unique} we can derive in the same way as in  \cite[Theorem~5.2]{LR} the following characterization for Hilbert's metric.
\begin{theorem}\label{t:unique_hilbert_geodesics}
If $A$ is a JB-algebra and $a,b \in A_+^\circ$ are linearly independent, then there exists a unique geodesic between $\ol{a}$ and $\ol{b}$  in $(\ol{A}_+^\circ, d_H)$ if and only
if $\sigma(U_{a^{-\frac{1}{2}}}b) = \{ \alpha, \beta \}$ for some $\beta > \alpha > 0$.
\end{theorem}
Recall that the straight line segment $\{\ol{(1-t)a+tb}\colon 0\leq t\leq 1\}$ is a geodesic in $(\ol{A}_+^\circ,d_H)$ for all $a,b\in A_+^\circ$.

The following special geodesic paths play an important role.
\begin{definition}
For $a,b \in \CA$, define the path $\gamma_a^b \colon [0,1] \to \CA$ by
\[ \gamma_a^b(t) := U_{a^{\frac{1}{2}}} \left(U_{a^{-\frac{1}{2}}}b \right)^t. \]
\end{definition}
Note that $\gamma_a^b(0) = U_{a^{\frac{1}{2}}}e = a$ and $\gamma_a^b(1) = U_{a^{\frac{1}{2}}} U_{a^{-\frac{1}{2}}}b = b$.
Also note that for $\lambda,\mu >0$ and $a,b\in \CA$,
\[
\ol{\gamma_{\lambda a}^{\mu b}(t)} = \ol{\gamma_{a}^{b}(t)}\mbox{\quad for all }t\in [0,1].
\]
Thus, we can define for $\ol{a},\ol{b}\in \ol{A}_+^\circ$ a path in $\ol{A}_+^\circ$ by $\gamma_{\ol{a}}^{\ol{b}}(t) := \ol{\gamma_a^b(t)}$ for all $t\in[0,1]$.

We will verify that $\gamma_a^b$ is a geodesic path connecting $a$ and $b$ in $(\CA,d_T)$. The argument to show that $\gamma_{\ol{a}}^{\ol{b}}$ is a geodesic in $(\ol{A}_+^\circ, d_H)$ is similar and is left to the reader.
Using the fact that $U_{c^\lambda} c^\mu = c^{2 \lambda + \mu}$ in the fourth step, we get that
\begin{eqnarray*}
d_T(\gamma_a^b(s), \gamma_a^b(t)) &= & d_T\left(U_{a^{\frac{1}{2}}} \left(U_{a^{-\frac{1}{2}}}b \right)^s, U_{a^{\frac{1}{2}}} \left(U_{a^{-\frac{1}{2}}}b \right)^t \right) \\
&= &d_T\left(\left(U_{a^{-\frac{1}{2}}}b \right)^s,  \left(U_{a^{-\frac{1}{2}}}b \right)^t \right) \\
&= &\norm{ \log U_{(U_{a^{-\frac{1}{2}}}b)^{-\frac{t}{2}}} (U_{a^{-\frac{1}{2}}}b)^s}\\
&= &\norm{ \log (U_{a^{-\frac{1}{2}}}b)^{s-t}} \\
&= & |s-t| \norm{ \log U_{a^{-\frac{1}{2}}}b} \\
&= & |s-t| d_T(a,b)
\end{eqnarray*}
for all $s,t\in[0,1]$.

\subsection{Geometric means in JB-algebras}
The cone $A_+^\circ$ in a JB-algebra is a symmetric space, see Lawson and Lim \cite{LL2} and Loos \cite{Loos}.
Indeed, for $c\in A_+^\circ$ one can define maps $S_c\colon A_+^\circ\to A_+^\circ$ by
\[
S_c(a) := U_c a^{-1}\mbox{\quad for }a\in A_+^\circ.
\]
Clearly $S_c(c) =c$, and $S_c^2(a) = U_c(U_ca^{-1})^{-1} = U_c(U_{c^{-1}}a) =a$ for all $a\in A_+^\circ$. Moreover, by the fifth equation in (\ref{e:U_facts}) we see that
\[
S_{S_c(b)}(S_c(a))= U_{U_cb^{-1}}(U_c a^{-1})^{-1} = U_cU_{b^{-1}}U_c(U_{c^{-1}}a) = U_c(U_ba^{-1})^{-1}) = S_c(S_b(a))
\]
for all $a\in A_+^\circ$. The map $S_c$ is called the {\em symmetry around $c$}, see \cite{Loos}.

The equation $S_c(a)=b$ has a unique solution in $A_+^\circ$, namely $\gamma_a^b(1/2)$. Indeed, using (\ref{e:U_facts}) and taking the unique positive square root in the third step, we obtain the following equivalent identities:
\begin{eqnarray*}
U_c a^{-1}= b & \Longleftrightarrow &
U_{a^{-\frac{1}{2}}} U_c a^{-1} = U_{a^{-\frac{1}{2}}} b \\
  & \Longleftrightarrow& ( U_{a^{-\frac{1}{2}}} c)^2 = U_{a^{-\frac{1}{2}}} b \\
 & \Longleftrightarrow& U_{a^{-\frac{1}{2}}} c = \left( U_{a^{-\frac{1}{2}}} b \right)^\frac{1}{2} \\
& \Longleftrightarrow& c =  U_{a^{\frac{1}{2}}} \left( U_{a^{-\frac{1}{2}}}b \right)^\frac{1}{2}.
\end{eqnarray*}

\begin{definition}
For  $a,b \in \CA$ the unique solution of the equation $S_c(a) = b$ is called the \emph{geometric mean} of $a$ and $b$. It is denoted by $a \# b$, so
\[ a \# b := U_{a^{\frac{1}{2}}} \left( U_{a^{-\frac{1}{2}}}b \right)^\frac{1}{2} \in \CA. \]
\end{definition}
We remark that the equation  $S_c(b)=U_c b^{-1} = a$, which has the unique solution $c=b \# a$, is equivalent to the equation $S_c(a)= U_ca^{-1} =b$. Thus, $a \# b = b \# a$, and hence $S_{a\# b}(a)=b$ and $S_{a\# b}(b)=a$.
Note  also that, as $S_c(a) =a$ implies that $c = a\#a =a$, the map $S_c$ has a unique fixed point $c$ in $A_+^\circ$. Moreover, $S_c$ is an isometry under both Hilbert's metric and Thompson's metric on $A_+^\circ$, since it is the composition of two isometries.

The idea is now to show that the geometric means are preserved under bijective Hilbert's metric and Thompson's metric isometries. The proof relies on properties of the  maps $S_{a\# b}$ and the following
lemma. This lemma and its proof are similar to \cite[lemma p.\,3852]{Mol1}, the only difference being that we consider two metric spaces here.
\begin{lemma}\label{l:metric_space}
Let $M,N$ be metric spaces. Suppose that for each $x,y \in M$ there exists an element $z_{xy} \in M$, a bijective isometry $\psi_{xy} \colon M \to M$ and a constant $k_{xy} > 1$ such that
\begin{enumerate}
 \item $\psi_{xy}(x) = y, \quad \psi_{xy}(y) = x$; \\
 \item $\psi_{xy}(z_{xy}) = z_{xy}$; \\
 \item $d(u, \psi_{xy}(u)) \geq k_{xy} d(u, z_{xy}) \mbox{\quad for all }  u \in M$.
\end{enumerate}
Suppose $N$ satisfies the same requirements. If $\phi \colon M \to N$ is a bijective isometry, then
\[ \phi(z_{xy}) = z_{\phi(x)\phi(y)}. \]
\end{lemma}

Applying this lemma to the maps $S_{a\# b}$ we derive the following proposition for Thompson's metric.
\begin{proposition}\label{p:Tmidpoints}
If  $A$ and $B$ are JB-algebras and  $f \colon \CA \to \CB$ is a bijective Thompson's metric isometry, then
\[
f(a \# b) = f(a) \# f(b)\mbox{\quad for all } a,b\in \CA.
\]
\end{proposition}
\begin{proof}
For $a,b \in \CA$ or $a,b \in \CB$, we already saw that $S_{a\# b}$ is an isometry that satisfies the first two properties in Lemma \ref{l:metric_space}. To show the third property note that  by Proposition~\ref{p:hilbert_thompson_metrics},
\[
d_T(S_c(a),a) =  \norm{\log U_{a^{-\frac{1}{2}}} U_c a^{-1} }=  \norm{ \log \left( U_{a^{-\frac{1}{2}}} c \right)^2 } \\
=  2 \norm{ \log U_{a^{-\frac{1}{2}}} c } =  2 d_T(c,a).
\]
So, if we take  $k_{ab} := 2$, then all conditions of Lemma \ref{l:metric_space} are satisfied, and its application yields the proposition.
\end{proof}
To see that the same result holds for Hilbert's metric isometries on $\ol{A}_+^\circ$, we need to make a couple of observations. Firstly for $c\in \CA$, the map $S_c$ induces a well defined maps $S_{\ol{c}}$ on $\ol{A}_+^\circ$ by letting
$S_{\ol{c}}(\ol{a}):= \ol{S_c(a)}$.  Furthermore,  for $a,b\in \CA$ and $\lambda,\mu>0$ we have that the equation $U_c(\lambda a) =\mu b$ has unique solution $c=(\lambda a)\# (\mu b) = \sqrt{\lambda\mu}(a\# b)$.
Thus, the equation $U_{\ol{c}} \ol{a}^{-1}= \ol{U_ca^{-1}}=\ol{b}$ has a unique solution $\ol{a\#b}$ in $\ol{A}_+^\circ$ for $\ol{a},\ol{b}\in\ol{A}_+^\circ$, and we can define the projective geometric mean by $\ol{a}\#\ol{b} := \ol{a\#b}$ in $\ol{A}_+^\circ$. Note  that $\ol{a}\#\ol{b}=\gamma_{\ol{a}}^{\ol{b}}(1/2)$.  It is now straightforward  to verify that the Hilbert's metric isometries $S_{\ol{a}\#\ol{b}}$ on $\ol{A}_+^\circ$
satisfy the requirements of Lemma \ref{l:metric_space}  with $k_{ab} = 2$ and  derive the following result.
\begin{proposition}\label{p:Hmidpoints}
If  $A$ and $B$ are JB-algebras and  $f \colon \ol{A}_+^\circ \to \ol{B}_+^\circ$ is a bijective Hilbert's metric isometry, then
\[
f(\ol{a} \# \ol{b}) = f(\ol{a}) \# f(\ol{b})\mbox{\quad for all }\ol{a},\ol{b}\in \ol{A}_+^\circ.
\]
\end{proposition}
The next proposition will be useful.
\begin{proposition}\label{p:geo_midpoints}
For all $a,b \in \CA$ and $t,s \in [0,1]$,
\[ \gamma_a^b(t) \# \gamma_a^b(s) =  \gamma_a^b \left( \frac{t+s}{2} \right). \]
\end{proposition}
\begin{proof}
Using \eqref{e:U_facts}, the computation below shows that $c=\gamma((t+s)/2) $ is a positive solution of $U_c \gamma(t)^{-1} = \gamma(s)$, which proves the proposition.
\begin{eqnarray*}
U_{\gamma(\frac{t+s}{2})} \gamma(t)^{-1} &= & U_{U_{a^{\frac{1}{2}}} (U_{a^{-\frac{1}{2}}}b)^{\frac{t+s}{2}}} \left( U_{a^{\frac{1}{2}}} (U_{a^{-\frac{1}{2}}}b)^t \right)^{-1} \\
&= & U_{a^{\frac{1}{2}}} U_{(U_{a^{-\frac{1}{2}}}b)^{\frac{t+s}{2}}} U_{a^{\frac{1}{2}}} U_{a^{-\frac{1}{2}}} (U_{a^{-\frac{1}{2}}}b)^{-t} \\
&= & U_{a^{\frac{1}{2}}} U_{(U_{a^{-\frac{1}{2}}}b)^{\frac{t+s}{2}}} (U_{a^{-\frac{1}{2}}}b)^{-t} \\
&= & U_{a^{\frac{1}{2}}} (U_{a^{-\frac{1}{2}}}b)^{s} \\
&= & \gamma(s).
\end{eqnarray*}
\end{proof}
It is straightforward to derive a similar identity for Hilbert's metric.
\begin{proposition}\label{p:geo_midpoints_H}
For all $\ol{a},\ol{b} \in \ol{A}_+^\circ$ and $t,s \in [0,1]$,
\[ \gamma_{\ol{a}}^{\ol{b}}(t) \# \gamma_{\ol{a}}^{\ol{b}}(s) =  \gamma_{\ol{a}}^{\ol{b}} \left( \frac{t+s}{2} \right). \]
\end{proposition}
\begin{proof}
The proof follows from Proposition \ref{p:geo_midpoints} and
\[ \gamma_{\ol{a}}^{\ol{b}}(t) \# \gamma_{\ol{a}}^{\ol{b}}(s) = \ol{\gamma_a^b(t)} \# \ol{\gamma_a^b(s)}=  \ol{\gamma_a^b(t) \# \gamma_a^b(s)} = \ol{ \gamma_a^b \left( \frac{t+s}{2} \right)} =  \gamma_{\ol{a}}^{\ol{b}} \left( \frac{t+s}{2} \right).\]
\end{proof}
By combining Propositions \ref{p:Tmidpoints}  and \ref{p:geo_midpoints} we derive the following corollary. The proof  uses the fact that the equation $a\# c=b$ has a unique solution $c=U_ba$, which can be easily shown using  \eqref{e:U_facts}.
\begin{corollary}\label{c:exp_commutes_with_powers_T}
Let $A$ and $B$ be JB-algebras. If $f \colon \CA \to \CB$ is a bijective Thompson's metric isometry, then
\begin{enumerate}[(a)]
\item $f$ maps $\gamma_a^b(t)$ to $\gamma_{f(a)}^{f(b)}(t)$ for all $a,b\in \CA$ and $t\in [0,1]$.
\item If $f(e)=e$, then  $f(a^t) = f(a)^t$ for all $t\in[0,1]$. Moreover, we have $f(a^{-1}) = f(a)^{-1}$ and $f(U_b a) = U_{f(b)} f(a)$.
\end{enumerate}
\end{corollary}
\begin{proof}
By Propositions~\ref{p:geo_midpoints} and \ref{p:Tmidpoints}, the first statement holds for all dyadic rationals $t \in [0,1]$. As the dyadic rationals are dense in $[0,1]$, it holds for all $0 \leq t \leq 1$.

Suppose $f(e)=e$. Since $\gamma_e^a(t) = a^t$, the first statement yields that $f(a^t) = f(a)^t$ for all $0 \leq t \leq 1$.

Since
$$ a \# a^{-1} = U_{a^{\frac{1}{2}}} ( U_{a^{-\frac{1}{2}}} a^{-1} )^\frac{1}{2} = U_{a^{\frac{1}{2}}} a^{-1} = e, $$
we have that $f(a) \# f(a^{-1}) = f(a \# a^{-1}) = f(e) = e = f(a) \# f(a)^{-1}$, so by uniqueness of the solution of $f(a) \# c = e$, we obtain $f(a^{-1}) = f(a)^{-1}$. Using  \eqref{e:U_facts} we also get
$$ f(a)^{-1} \# f(U_b a) = f(a^{-1} \# U_b a) = f(U_{a^{-\frac{1}{2}}} (U_{a^\frac{1}{2}} U_b a)^\frac{1}{2} ) = f(U_{a^{-\frac{1}{2}}} U_{a^\frac{1}{2}} b) = f(b), $$
 so $f(b)$ is a solution to $S_c(f(a)^{-1}) = f(U_ba)$, i.e., $U_{f(b)} f(a) = f(U_b a)$.

\end{proof}
Again, a similar result holds for Hilbert's metric. The proof is analogous to the one for Thompson's metric in Corollary \ref{c:exp_commutes_with_powers_T} and is left to the reader.
\begin{corollary}\label{c:exp_commutes_with_powers_H}
Let $A$ and $B$ be JB-algebras. If $f \colon \ol{A}_+^\circ \to \ol{B}_+^\circ$ is a bijective Hilbert's metric isometry, then
\begin{enumerate}[(a)]
\item $f$ maps $\gamma_{\ol{a}}^{\ol{b}}(t)$ to $\gamma_{f(\ol{a})}^{f(\ol{b})}(t)$ for all $\ol{a},\ol{b}\in \ol{A}_+^\circ$ and $t\in [0,1]$.
\item If $f(\ol{e})=\ol{e}$, then  $f(\ol{a}^t) = f(\ol{a})^t$ for all $t\in[0,1]$. Moreover, we have $f(\ol{a}^{-1}) = f(\ol{a})^{-1}$ and $f(U_{\ol{b}}\ol{a}) = U_{f(\ol{b})} f(\ol{a})$.
\end{enumerate}
\end{corollary}
Now we can prove an essential ingredient for characterizing bijective Hilbert's metric and Thompson's metric isometries of cones of JB-algebras. Recall that $[A] = A / \Span(e)$.
\begin{theorem}\label{t:inducing_isometry}
Let $A$ and $B$ be JB-algebras.
\begin{enumerate}[(a)]
\item If $f \colon \CA \to \CB$ is a bijective Thompson's metric isometry with $f(e) = e$, then $S \colon A \to B$ given by
\[ Sa := \log f( \exp(a)), \]
is a bijective linear $\norm{\cdot}$-isometry.
\item If $f\colon \ol{A}_+^\circ\to \ol{B}_+^\circ $ is a bijective Hilbert's isometry with $f(\ol{e}) = \ol{e}$, then $S \colon [A] \to [B]$ given by
\[ S[a] := \log f( \exp([a])),\]
is a bijective linear $\norm{\cdot}_v$-isometry.
\end{enumerate}
\end{theorem}
\begin{proof}
We will prove the second assertion. The same arguments can be used to show the statements for Thompson's metric.
Using Corollary~\ref{c:exp_commutes_with_powers_H},
\[
\exp(S[a]/n) =  \exp (\log(f( \exp([a])))/n)
= \exp (\log f(\ol{\exp(a)})^{1/n})
= f( \ol{\exp(a)})^{1/n}
= f( \ol{\exp(a/n)}).
\]
Thus,
\begin{align*}
d_n^H(S[a], S[b]) &= n d_H( \exp(S[a]/n), \exp(S[b]/n))\\&=n d_H(f( \ol{\exp(a/n)}), f( \ol{\exp(b/n)})) \\
&= n d_H(\exp(a/n), \exp(b/n))\\&=d^H_n([a],[b]).
\end{align*}
By Proposition~\ref{p:d_n_converges_to_norm} the left-hand side of the above equation converges to $\norm{S[a]-S[b]}_v$ and the right-hand side converges to $\norm{[a]-[b]}_v$ as $n\to\infty$. Hence $S$ is a bijective $\norm{\cdot}_v$-isometry. As $f(\ol{e})= \ol{e}$, we have that $S[0] = [0]$, and hence $S$ is linear by the Mazur-Ulam theorem.
\end{proof}

\begin{remark}
The map $\Exp \colon A \to \CA$ is a bijection. In the associative case, where $A = C(K)$ for some compact Hausdorff space $K$, one can show that this bijection induces an isometric isomorphism between the spaces $(A, \norm{\cdot})$ and $(\CA, d_T)$, see \cite{LRW1}. Likewise, the exponential map yields an  isometric isomorphism between $([A], \norm{\cdot}_v)$ and $(\ol{A}_+^\circ,d_H)$ if $A=C(K)$.  In the nonassociative case this is no longer true. In fact, it has been shown for finite dimensional order unit spaces $A$ that $(\ol{A}_+^\circ,d_H)$ is isometric to a normed space if and only if $A_+$ is a simplicial cone, see \cite{FKa}. For Thompson's metric the same result  holds, see \cite[Theorem 7.7]{LR}.
\end{remark}

\section{Thompson's metric isometries of JB-algebras}\label{sec:thomp}

The next basic property of Thompson's metric on products of cones will be useful.
\begin{proposition}\label{p:thompson_product}
Suppose that $A$ is a product of order unit space $A_i$ for $i \in I$.  If $d_T^i$  denotes the Thompson's metric on $A_{i+}^\circ$ and $a= (a_i), b = (b_i)\in \CA$, then
\[ d_T(a,b) = \sup_{i \in I} d_T^i(a_i,b_i). \]
\end{proposition}
\begin{proof}
The proposition follows immediately from
\[
M_A(a/b) = \inf \{ \lambda > 0 \colon a_i\leq \lambda b_i \mbox{ for all } i \in I\}
= \sup_{i \in I} \inf \{ \lambda > 0\colon a_i\leq \lambda b_i \}
= \sup_{i \in I} M_{A_i}(a_i, b_i).
\]
\end{proof}

With the above preparations we can now obtain the following theorem.
The proof, as well as the statement, is a direct generalization of \cite[Section 4]{Bo} and \cite[Theorem 9]{Mol1}.
\begin{theorem}\label{t:thompson_isometry}
Let $A$ and $B$ be unital JB-algebras. A map $f \colon \CA \to \CB$ is a bijective Thompson's metric isometry if and only if there exist $b \in B_+^\circ$, a central projection $p \in B$, and a Jordan isomorphism $J \colon A \to B$ such that $f$ is of the form
\[ f(a) = U_b (pJa + p^\perp J a^{-1}) \mbox{\quad for all } a \in \CA. \]
In this case $b = f(e)^{\frac{1}{2}}$.
\end{theorem}
\begin{proof}
The last statement follows from taking $a=e$, which yields $b^2 = f(e)$.

For the sufficiency, note that the central projection $p$ yields a decomposition $B = pB \oplus p^\perp B$, which is left invariant by $U_b$. This decomposition can be pulled back by $J$, which yields the following representation of the map $f \colon (J^{-1}pB)_+^\circ \times (J^{-1}p^\perp B)_+^\circ \to (pB)_+^\circ \times (p^\perp B)_+^\circ$:
\[ f(a_1, a_2) = (U_{b} J a_1, U_{b}J a_2^{-1}). \]
Note that a Jordan isomorphism is an order isomorphism and hence an isometry under Thompson's metric. The inversion and the quadratic representations also preserve  Thompson's metric, and so Thompson's metric is preserved on both parts. By Proposition~\ref{p:thompson_product} Thompson's metric is preserved on the product as well.

Now suppose that $f\colon \CA \to \CB$ is a bijective Thompson's metric isometry. Defining $g(a) := U_{f(e)^{-\frac{1}{2}}} f(a)$, we obtain that $g$ is a Thompson's metric isometry mapping $e$ to $e$. By Theorem~\ref{t:inducing_isometry} the map $S \colon A \to B$ defined by
\[ Sa := \log g( \exp(a)) \]
is a bijective linear $\norm{\cdot}$-isometry.

From Theorem~\ref{t:jbisoms} it follows that there is a central projection $p \in B$ and a Jordan isomorphism $J \colon A \to B$ such that $Sa = (p-p^\perp)Ja$. We now have for $a \in A$,
\begin{align*}
g(\exp(a)) &= \exp(Sa) = \exp((p-p^\perp)Ja) \\
&= \sum_{n=0}^\infty \frac{(p-p^\perp)^n (Ja)^n}{n!} \\
&= \sum_{n=0}^\infty \frac{(p + (-1)^n p^\perp) J(a^n)}{n!} \\
&= p J \left( \sum_{n=0}^\infty \frac{a^n}{n!} \right) + p^\perp J \left( \sum_{n=0}^\infty \frac{(-a)^n}{n!} \right) \\
&= pJ( \exp(a)) + p^\perp J(\exp(-a)).
\end{align*}
It follows that, for $a \in \CA$, $g(a) = pJa + p^\perp Ja^{-1}$. The theorem now follows from
\[ f(a) = U_{f(e)^{\frac{1}{2}}} U_{f(e)^{-\frac{1}{2}}} f(a) = U_{f(e)^{\frac{1}{2}}} g(a). \]
\end{proof}

\subsection{The Thompson's metric isometry group of a JB-algebra}
In the case where a JB-algebra is the direct product of simple JB-algebras, we can explicitly compute its Thompson's metric  isometry group in terms of the Jordan automorphism groups of the simple components. Each Euclidean Jordan algebra satisfies this requirement, and the automorphism groups of the simple Euclidean Jordan algebras are known, see \cite{FK}.

\begin{theorem}\label{t:d_T_isom_group_JB}
Suppose a JB-algebra $A$ can be decomposed as a direct product
\[ A = \prod_{i \in I} A_i^{n_i}, \]
where $I$ is an index set, the $n_i$ are arbitrary cardinals and the $A_i$ are mutually nonisomorphic simple JB-algebras. Then the Thompson's metric isometry group of $A$ equals
\[ \Isom(A_+^\circ, d_T) = \prod_{i \in I}  \left( \Aut(A_{i+}) \rtimes C_2\right)^{n_i}  \rtimes S(n_i), \]
where $\Aut(A_{i+})$ denotes the automorphism group of the cone $A_{i+}$, i.e., the order isomorphisms of $A_i$ into itself, $C_2$ denotes the cyclic group of order 2 generated by the inverse map $\iota$, and $S(n_i)$ denotes the group of permutations of $n_i$.
\end{theorem}
\begin{proof}
By Theorem~\ref{t:thompson_isometry} any bijective Thompson's metric isometry is a composition of a quadratic representation, a Jordan isomorphism and taking inverses on some components. Quadratic representations and taking inverses leave each component invariant, and Jordan isomorphisms leave the Jordan isomorphism classes invariant. This shows that
\[
\mathrm{Isom}(A_+^\circ, d_T) \subseteq \prod_{i \in I} \mathrm{Isom}((A^{n_i}_i)^\circ_+, d_T),
\]
and the other inclusion follows from Proposition~\ref{p:thompson_product}, so we have equality. We will now investigate $\mathrm{Isom}((A^{n_i}_i)^\circ_+, d_T)$.

A Jordan isomorphism of $A^{n_i}_i$ may permute the components, so it follows that each Thompson's metric isometry of $(A^{n_i}_i)^\circ_+$ is a composition of a permutation of components, a componentwise possible inversion, a componentwise Jordan isomorphism, and a componentwise quadratic representation. So, all the operators except the permutation will act componentwise, and the componentwise operators form a subgroup. It is easy to compute that a componentwise operator conjugated by a permutation $\pi$ equals the componentwise operator permuted by $\pi$. This shows that the componentwise operators and the permutation group form a semidirect product, where the componentwise operators are the normal subgroup. It remains to examine the componentwise operators.

By Proposition~\ref{p:order_isomorphism}, any order isomorphism is the product of a quadratic representation and a Jordan isomorphism. If we denote the inverse map by $\iota = \iota^{-1}$, then conjugating an order isomorphisms with the inverse map gives
\begin{equation}\label{e:action_inverse}
 (\iota U_b J \iota^{-1}) (a) = (U_b Ja^{-1})^{-1} = U_{b^{-1}} (Ja^{-1})^{-1} = U_{b^{-1}} Ja,
\end{equation}
which yields another order isomorphism. So, the product of the group of order isomorphism and the inversion group $C_2$ is a semidirect product, where the order isomorphisms form the normal subgroup. We conclude that
\[
\mathrm{Isom}(A_+^\circ, d_T)=\prod_{i \in I} \mathrm{Isom}((A^{n_i}_i)^\circ_+, d_T)\cong\prod_{i \in I} \left( \mathrm{Aut}(A_{i+}) \rtimes C_2\right)^{n_i}  \rtimes S(n_i).
\]
\end{proof}

\begin{remark}
If $A$ is a JB-algebra as given in the above theorem, then we can use an analogous argument to show that the automorphism group of the cone $A_+$ equals
\[
\mathrm{Aut}(A_+)=\prod_{i\in I}\mathrm{Aut}(A_{i+}^{n_i})=\prod_{i \in I} \mathrm{Aut}(A_{i+})^{n_i} \rtimes S(n_i).
\]
Furthermore, for any $i\in I$ the conjugation action \eqref{e:action_inverse} on an order isomorphism in $\mathrm{Aut}(A_{i+}^{n_i})$ also shows that
\[
\mathrm{Isom}((A_i^{n_i})_+^\circ,d_T)\cong\mathrm{Aut}(A_{i+}^{n_i})\rtimes C_2^{n_i},
\]
so we can write the isometry group as
\[
\mathrm{Isom}(A_+^\circ,d_T)\cong\prod_{i\in I}\mathrm{Aut}(A_{i+}^{n_i})\rtimes C_2^{n_i}.
\]
It follows that the automorphism group $\mathrm{Aut}(A_+)$ is normal in $\mathrm{Isom}(A_+^\circ,d_T)$, and its quotient is isomorphic to $\prod_{i\in I}C_2^{n_i}$.
Suppose now that both $I$ and $n_i$ are finite (i.e., $A$ is a Euclidean Jordan algebra). Then the index of the automorphism group in the isometry group for Thompson's metric is $2^m$, where $m = \sum_{i \in I} n_i$ is the total number of different components. This is a correction of \cite[Remark~4.9]{Bo}, which has the wrong index.
\end{remark}

\section{Hilbert's metric isometries of JBW-algebras}\label{sec:hilbert}

If $A$ and $B$ are JB-algebras and $f \colon \ol{A}_+^\circ \to \ol{B}_+^\circ$ is a bijective Hilbert's metric isometry mapping $\ol{e}$ to $\ol{e}$, then by Theorem \ref{t:inducing_isometry} the map $S \colon [A] \to [B]$ defined by, $S[a] := \log f(\exp([a]))$, is a bijective linear $\norm{\cdot}_v$-isometry. Every bijective linear isometry maps extreme points of the unit ball to extreme points of the unit ball, which is what we will exploit here. Let us first identify these extreme points. For JBW-algebras this is \cite[Proposition~2.2]{Ham}.

\begin{lemma}\label{l:ext} The extreme points of the unit ball in $([A], \norm{\cdot}_v)$ are the equivalence classes $[p]$, where $p\in A$ is a nontrivial projection.
\end{lemma}
\begin{proof} Let $p\in A$ be a nontrivial projection and suppose that $[p]=t[a]+(1-t)[b]$ for some $0<t<1$, and
$[a], [b]\in [A]$ with $\|[a]\|_v=\|[b]\|_v=1$. There exist $\lambda\in\mathbb{R}$, $a\in[a]$, and $b\in[b]$ such that $p=ta+(1-t)b+\lambda e$ and
\[
\{0,1\}\subseteq\sigma(a),\sigma(b)\subseteq[0,1].
\]
This implies that
\[
\{-\lambda,1-\lambda\}=\sigma(p)-\lambda=\sigma(p-\lambda e)=\sigma(ta+(1-t)b)\subseteq[0,\|ta+(1-t)b\|]\subseteq[0,1],
\]
from which we conclude that $\lambda=0$. By the same argument as in \cite[Lemma 2.23]{AS}, the extreme points of those elements $a \in A$ with $\sigma(a)\subseteq[0,1]$ are  projections. So, $p=a=b$, and hence $[p]=[a]=[b]$, which shows that $[p]$ is an extreme point of the unit ball in $([A],\|\cdot\|_v)$.

Conversely, if $[a]\in[A]$ with $\|[a]\|_v=1$  does not contain a projection, then a representative $a$ with $\sigma(a)\subseteq[0,1]$ must have $\lambda\in\sigma(a)$ with $0<\lambda<1$. Now consider JB$(a,e) \cong C(\sigma(a))$. By elementary topology there exists a nonnegative function $g\in C(\sigma(a))$ with $g\neq 0$ such that the ranges of  $a+g$ and $a-g$ are contained in $[0,1]$.  Since $a=\frac{1}{2}(a-g)+\frac{1}{2}(a+g)$, it follows that $[a]$ can be written as $\frac{1}{2}([b]+[c])$ with $[b]\neq [c]$ and $\|[b]\|_v =\|[c]\|_v =1$, and hence $[a]$ cannot be an extreme point of the unit ball.
\end{proof}

To be able to exploit the extreme points  we will  restrict ourselves to cones in  JBW-algebras, as JB-algebras may not have nontrivial projections, e.g.\ $C([0,1])$. For a JBW-algebra $M$ we will denote its set of projections by $\P(M)$.

Let $M$ be a JBW-algebra. By Lemma \ref{l:ext} we can define a map $\theta \colon \P(M) \to \P(N)$ by letting $\theta(0)=0$, $\theta(e)=e$, and $\theta(p)$ be the unique nontrivial projection in the class $S[p]$, otherwise. Thus, for each bijective Hilbert's metric isometry $f\colon \ol{M}_+^\circ\to\ol{N}_+^\circ$ with $f(\ol{e})=\ol{e}$, we get a bijection $\theta\colon\P(M)\to\P(N)$. We say that $\theta$ is {\em induced} by $f$.  Note that its inverse $\theta^{-1}$ is induced by $f^{-1}$. The map $\theta$ will be the key in understanding $f$.

We call a bijection $\theta\colon \P(M)\to\P(N)$ an {\em orthoisomorphism} if $p,q\in\P(M)$ are orthogonal if and only if $\theta(p)$ and $\theta(q)$ are orthogonal. Our goal will be to prove that the map $\theta$ induced by either $f$ or $\iota \circ f$, where $\iota(\ol{a}) =\ol{a}^{-1}$ is the inversion, is in fact an orthoisomorphism. For this we need to investigate certain unique geodesics starting from the unit $e$.

We introduce the following notation: $(\ol{a},\ol{b})$ denotes the open line segment $\{t\ol{a} + (1-t)\ol{b}: 0 < t < 1\}$ in $\ol{M}_+$ for $\ol{a},\ol{b}\in \ol{M}_+$. The segments $[\ol{a},\ol{b}]$ and $[\ol{a},\ol{b})$ are defined similarly. Furthermore, we denote the affine span of a set $S$ by $\mathrm{aff}\,(S)$.
\begin{lemma}\label{ksimplices}
If $p_1,\ldots,p_k$ are nontrivial projections in a JBW-algebra $M$ such that $p_1+\cdots +p_k=e$, then the boundary of $\mathrm{conv}(p_1,\ldots,p_k)$ is contained in $\partial M_+$ and so
\[
\mathrm{aff}(p_1,\ldots,p_k)\cap M_+= \mathrm{conv}(p_1,\ldots,p_k),
\]
which is a $(k-1)$-dimensional simplex. Moreover, for each $a\in \mathrm{conv}(p_1,\ldots,p_k)\cap M_+^\circ$ the segment $[\ol{a},\ol{p}_i)$ is a unique geodesic in $(\ol{M}_+^\circ,d_H)$ for all $i=1,\ldots,k$.
\end{lemma}
\begin{proof}
As $p_1+\cdots +p_k=e$, it follows from \cite[Proposition~2.18]{AS} that the $p_i$ are pairwise orthogonal. So,
\[
0\in\sigma( \lambda_1p_1+\cdots +\lambda_kp_k)\mbox{\quad for }\lambda_1+\cdots+\lambda_k =1\mbox{ and }0\leq \lambda_i\leq 1\mbox{ for all }i=1,\ldots,k
\]
if and only if $\prod_{i=1}^k \lambda_i=0$. Hence the relative boundary of $\mathrm{conv}(p_1,\ldots,p_k)$ in $\mathrm{aff}(p_1,\ldots,p_k)$ lies in $\partial M_+$, which  proves the equality.

Note that if $a= \mu_1p_1+\cdots+\mu_kp_k$ with $\mu_1+\cdots+\mu_k =1$ and $0< \mu_i< 1$ for all $i=1,\ldots,k$, then $a^{-\frac{1}{2}}= \mu_1^{-\frac{1}{2}}p_1+\cdots+\mu_k^{-\frac{1}{2}}p_k$. Now let $b_i:=\frac{1}{2}(a+p_i)$. Then
\[
U_{a^{-\frac{1}{2}}} b_i = \frac{1}{2}(U_{a^{-\frac{1}{2}}} a+U_{a^{-\frac{1}{2}}} p_i) =\frac{1}{2}(e+\mu_i^{-1}p_i),
\]
and hence $\sigma(U_{a^{-\frac{1}{2}}} b_i)=\{\frac{1}{2},\frac{1+\mu_i^{-1}}{2}\}$. So, it follows from Theorem \ref{t:unique_hilbert_geodesics} that  $[\ol{a},\ol{p}_i)$ is a unique geodesic in $(\ol{M}_+^\circ,d_H)$ for all $i=1,\ldots,k$.
\end{proof}

\begin{lemma}\label{l:proj_to_proj}Let $M$ and $N$ be JBW-algebras and $f\colon \ol{M}_+^\circ\to\ol{N}_+^\circ$ be a bijective Hilbert's metric isometry with $f(\ol{e})=\ol{e}$. Let $p\in\mathcal{P}(M)$ be nontrivial. The geodesic segment $[\ol{e},\ol{p})$ is mapped to the geodesic segment $[\ol{e},\ol{q})$ by $f$ for some $q\in\mathcal{P}(N)$. Moreover, $S[p]=[q]$ and so $\theta(p)=q$.
\end{lemma}

\begin{proof}
The  geodesic segments $[\ol{e},\ol{p})$ is unique by Lemma~\ref{ksimplices}. Thus, $f([\ol{e},\ol{p}))$ is also a unique geodesic segments starting at $\ol{e}$, since $f(\ol{e})=\ol{e}$.

Now fix $0<t<1$ and let $b\in f(\ol{tp+(1-t)e})$. By Theorem  \ref{t:unique_hilbert_geodesics},
$\sigma(b) =\{\alpha,\beta\}$ with $\beta>\alpha>0$. Note that $b':=b-\alpha e\in\partial N_+\setminus\{0\}$.  Clearly, $\sigma(b')=\{0,\beta-\alpha\}$, and hence $b'\in [r]$ for some nontrivial projection $r\in \P(N)$. Note also that
\[
b = (1+\alpha)\left((1+\alpha)^{-1}b' +(1-(1+\alpha)^{-1})e\right),
\]
and hence the image of the $[\ol{e},\ol{p})$ under $f$ is $[\ol{e},\ol{r})$.

If $q$ is a nontrivial projection and $0< t<1$, then by using Proposition~\ref{p:hilbert_thompson_metrics} it is easy to verify that $d_H(tq+(1-t)e,e) =-\log (1-t)$. As $f$ is an isometry that fixes $\ol{e}$,  we find that
\begin{align}\label{e:geod}
f(\ol{tp+(1-t)e}) =\ol{tr+(1-t)e}
\end{align}
for all $0\le t<1$.
Using the spectral decomposition $p=1 p + 0 p^\perp$, we  now deduce that
\begin{align*}
S[p]&=\log f( \ol{\exp(1) p + \exp(0) p^\perp})=\log f(\ol{\exp(-1)e+(1-\exp(-1))p})\\&=\log(\ol{\exp(-1)e+(1-\exp(-1))r})=[\log(r+\exp(-1) r^\perp)] \\
&=[-r^\perp]=[r],
\end{align*}
and hence $q :=\theta(p)= r$.
\end{proof}

We can now show that $\theta$ preserves operator commuting projections.

\begin{proposition}\label{p:T_op_com}
If $p,q\in\mathcal{P}(M)$ operator commute, then $\theta(p),\theta(q)\in \mathcal{P}(N)$ operator commute.
\end{proposition}

\begin{proof}If $p$ and $q$ operator commute, then $e+p$ and $e+q$ operator commute. It follows that $U_{(e+p)^{1/2}}(e+q)=U_{(e+q)^{1/2}}(e+p)$, so $U_{\ol{e+p}^{1/2}}\ol{e+q}=U_{\ol{e+q}^{1/2}}\ol{e+p}$. By Corollary~\ref{c:exp_commutes_with_powers_H} and equation \eqref{e:geod} in the proof of Lemma~\ref{l:proj_to_proj},
\begin{align}\label{eq:quad_proj}
U_{\ol{e+\theta(p)}^{1/2}}\ol{e+\theta(q)}&=U_{f(\ol{e+p})^{1/2}}f(\ol{e+q})=f(U_{\ol{e+p}^{1/2}}\ol{e+q})=f(U_{\ol{e+q}^{1/2}}\ol{e+p})=U_{f(\ol{e+q})^{1/2}}f(\ol{e+p})\nonumber\\
&=U_{\ol{e+\theta(q)}^{1/2}}\ol{e+\theta(p)}.
\end{align}
The JB-algebra generated by $e+\theta(p)$, $e+\theta(q)$, and $e$ is a JC-algebra by \cite[Theorem 7.2.5]{HO}. So, we can think of $U_{(e+\theta(p))^{1/2}}(e+\theta(q))$ and $U_{(e+\theta(q))}(e+\theta(p))$ as
\[
(e+\theta(p))^{\frac{1}{2}}(e+\theta(q))(e+\theta(p))^{\frac{1}{2}}\quad\mbox{ and }\quad (e+\theta(q))^{\frac{1}{2}}(e+\theta(p))(e+\theta(q))^{\frac{1}{2}}
\]
respectively, for some $C$*-algebra multiplication. The equality in \eqref{eq:quad_proj} implies that
\[
(e+\theta(p))^{\frac{1}{2}}(e+\theta(q))(e+\theta(p))^{\frac{1}{2}}=\lambda(e+\theta(q))^{\frac{1}{2}}(e+\theta(p))(e+\theta(q))^{\frac{1}{2}}
\]
for some $\lambda>0$. Since
\[
\sigma((e+\theta(p))^{\frac{1}{2}}(e+\theta(q))(e+\theta(p))^{\frac{1}{2}})=\sigma((e+\theta(q))^{\frac{1}{2}}(e+\theta(p))(e+\theta(q))^{\frac{1}{2}})\subseteq(0,\infty),
\]
we must have $\lambda=1$. Let $a:=(e+\theta(p))^{\frac{1}{2}}(e+\theta(q))^{\frac{1}{2}}$. This element satisfies the identity $a(e+\theta(p))^{\frac{1}{2}}=(e+\theta(p))^{\frac{1}{2}}a^*$, so by the Fuglede-Putnam theorem \cite[Theorem IX.6.7]{Con}, we find that $a^*(e+\theta(p))^{\frac{1}{2}}=(e+\theta(p))^{\frac{1}{2}}a$. This implies that
\begin{align*}
(e+\theta(p))(e+\theta(q))&=((e+\theta(p))(e+\theta(q))^{\frac{1}{2}})(e+\theta(q))^{\frac{1}{2}}=(e+\theta(q))^{\frac{1}{2}}((e+\theta(p))(e+\theta(q))^{\frac{1}{2}})\\&=(e+\theta(q))^{\frac{1}{2}}((e+\theta(q))^{\frac{1}{2}}(e+\theta(p)))
=(e+\theta(q))(e+\theta(p));
\end{align*}
hence $\theta(p)\theta(q)=\theta(q)\theta(p)$. So, $\theta(p)$ and $\theta(q)$ operator commute in JB$(\theta(p),\theta(q),e)$ by \cite[Proposition 1.49]{AS}, and therefore $\theta(p)$ and $\theta(q)$ generate an associative algebra. We conclude that $\theta(p)$ and $\theta(q)$ must operator commute in $N$ by \cite[Proposition 1.47]{AS}.
\end{proof}

This allows us to show that $\theta$ preserves orthogonal complements.

\begin{lemma}\label{lem:T_orth_complements} $\theta(p^\perp)=\theta(p)^\perp$ for all $p\in\P(M)$.
\end{lemma}

\begin{proof}
We may assume that $p$ is nontrivial by definition of $\theta$. Since $S[p]+S[p^\perp] =S[e] =[e]$, we obtain $\theta(p)+\theta(p^\perp)=\lambda e$ for some $\lambda \in\mathbb{R}$. As $p$ and $p^\perp$ operator commute, the projections $\theta(p)$ and $\theta(p^\perp)$ operator commute by Proposition~\ref{p:T_op_com}. By \cite[Proposition 1.47]{AS}, $\theta(p)$ and $\theta(p^\perp)$ are contained in an associative subalgebra, which is isomorphic to a $C(K)$-space.  In a $C(K)$-space it is obvious that $\lambda =1 $ or $\lambda =2$. Now note that $\lambda=2$ implies that  $\theta(p)=\theta(p^\perp)=e$ which contradicts the injectivity of $S$, and hence $\theta(p)+\theta(p^\perp)=e$, which shows that $\theta(p^\perp)=\theta(p)^\perp$.
\end{proof}

We will proceed to show that if  $f\colon \ol{M}_+^\circ\to\ol{N}_+^\circ$ is a bijective Hilbert's metric isometry with $f(\ol{e})=\ol{e}$, then for either $f$ or $ \iota\circ f$, the induced map $\theta$ maps orthogonal noncomplementary projections to orthogonal projections. For this we need to look at special simplices in the cone $M_+$.

\subsection{Orthogonal simplices}

Given  nontrivial projections $p_1,p_2,p_3$ in a JBW-algebra $M$ with  $p_1+p_2+p_3=e$, we call
\[
\Delta(p_1,p_2,p_3) :=\ol{\mathrm{conv}\,(p_1,p_2,p_3)\cap M_+^\circ}
\]
an {\em orthogonal simplex} in $\ol{M}_+^\circ$. The next lemma  shows that a bijective Hilbert's metric isometry $f$ maps orthogonal simplices onto orthogonal simplices.
\begin{lemma}\label{triangles}
Let $f\colon \ol{M}_+^\circ\to \ol{N}_+^\circ$ be a bijective Hilbert's metric isometry with $f(\ol{e}) =\ol{e}$. If $\Delta(p_1, p_2, p_3)$ is an  orthogonal simplex and $q_i = \theta(p_i)$ for $i=1,2,3$, then
\begin{enumerate}
\item $q_1 + q_2 + q_3 = e$, and then $f(\Delta(p_1,p_2,p_3))= \Delta(q_1,q_2,q_3)$, or
\item $q_1^\perp + q_2^\perp + q_3^\perp = e$, and then $f(\Delta(p_1,p_2,p_3))= \Delta(q^\perp_1,q^\perp_2,q^\perp_3)$.
\end{enumerate}
In case $(i)$, $\theta$ preserves the orthogonality of $p_1,p_2,p_3$. Moreover, if the map $\theta$ induced by $f$ satisfies the assumptions of case $(ii)$, then the map $\theta$ induced by the isometry  $\iota\circ f$ satisfies the conditions of case $(i)$.
\end{lemma}

\begin{proof}
First remark that, as $p_1+p_2+p_3=e$ and $S$ is linear, $S[p_1]+S[p_2]+S[p_3] =S[e] =[e]$, and hence
\begin{equation}\label{T1}
q_1 + q_2 + q_3 = \theta(p_1)+\theta(p_2)+\theta(p_3) =\lambda e\mbox{\quad for some }\lambda\in\mathbb{R}.
\end{equation}

As $p_1+p_2<e$, we know that $p_1$ and $p_2$ are orthogonal by  \cite[Proposition~2.18]{AS}, and hence $p_1$ and $p_2$ operator commute by  \cite[Proposition 1.47]{AS}.  We also know from Proposition \ref{p:T_op_com} that $q_1 = \theta(p_1)$ and $q_2 = \theta(p_2)$ operator commute. By \cite[Proposition 1.47]{AS}, $q_1$ and $q_2$ are contained in an associative subalgebra, which is isomorphic to a $C(K)$-space. Note that this subalgebra also contains $\lambda e$ and hence also $q_3$ by (\ref{T1}). In a $C(K)$-space it is obvious that $\lambda \in \{1, 2\}$ in \eqref{T1}. In fact, the case $\lambda = 1$ corresponds with the pairwise orthogonality of $q_1$, $q_2$ and $q_3$, whereas the case $\lambda =2$ corresponds to pairwise orthogonality of $q_1^\perp$, $q_2^\perp$ and $q_3^\perp$, and $q_1^\perp+q_2^\perp+q_3^\perp = e$.

We will now show that $f$ maps $\Delta(p_1,p_2,p_3)$ onto $\Delta(q_1,q_2,q_3)$ in case $q_1+q_2+q_3 = e$.  Let $a\in\mathrm{conv}(p_1,p_2,p_3)\cap M_+^\circ$ be a point not lying on any $(p_i,p_i^\perp)$ for $i=1,2,3$. We know that $[\ol{a},\ol{p}_i)$ is a unique geodesic by Lemma \ref{ksimplices}.
Let $(\ol{a}',\ol{p}_1)$ be the line segment through $\ol{p}_1$ and $\ol{a}$ with $a'$ in the boundary of $\mathrm{conv}(p_1,p_2,p_3)$. This unique geodesic intersects $(\ol{p}_2,\ol{p}_2^\perp)$ and $(\ol{p}_3,\ol{p}_3^\perp)$ in 2 distinct points, say $\ol{b}_2$ and $\ol{b}_3$ respectively, see Figure 4.
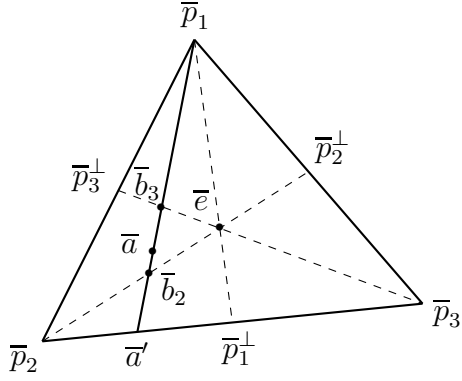
\begin{figure}[h]
\setcounter{figure}{3}
\[\begin{tikzpicture}
\draw[thick](0,0)--(2,4);
\draw[thick](0,0)--(5,0.5);
\node[below left]at(0.1,0.1){$\ol{p}_2$};
\draw[thick](2,4)--(5,0.5);
\node[above left]at(2.35,4){$\ol{p}_1$};
\node[below right]at(5,0.7){$\ol{p}_3$};
\draw[dashed](0,0)--(3.5,2.25);
\node[above right]at(3.45,2.20){$\ol{p}_2^\perp$};
\draw[dashed](2,4)--(2.5,0.25);
\node[below]at(2.6,0.28){$\ol{p}_1^\perp$};
\draw[dashed](1,2)--(5,0.5);
\node[left]at(1,2.2){$\ol{p}_3^\perp$};
\node[above left]at(2.33,1.6){$\ol{e}$};
\node at(2.33,1.51){\tiny{$\bullet$}};

\draw[thick](2,4)--(1.25,0.125);
\node[below]at(1.25,0.125){$\ol{a}'$};
\node[left]at(1.4,1.3){$\ol{a}$};
\node at (1.45,1.2) {\tiny{$\bullet$}};

\node[left]at(1.7,2.1){$\ol{b}_3$};
\node at(1.56,1.78) {\tiny{$\bullet$}};

\node[left]at(2.05,.7){$\ol{b}_2$};
\node at (1.4,.90){\tiny{$\bullet$}};

\end{tikzpicture}\]
\caption{Orthogonal simplex}
\label{simplex}
\end{figure}
Since it must be mapped to a line segment, it follows that $f(\ol{a})$ lies on the line segment through  $f(\ol{b}_2)$ and $f(\ol{b}_3)$, which is contained in  $\Delta(q_1,q_2,q_3)$.  By the invertibility of $f$, we conclude that $f(\Delta(p_1,p_2,p_3))= \Delta(q_1,q_2,q_3)$. The same argument can be used to show that  $f(\Delta(p_1,p_2,p_3))= \Delta(q^\perp_1,q^\perp_2,q^\perp_3)$ if $q_1^\perp+q_2^\perp+q_3^\perp=e$.

To prove the final statement remark that if  we compose $f$ with the inversion $\iota$, we obtain
\[
S [p_i] = \log \iota(f( \exp([p_i]))) = \log f( \exp([p_i]))^{-1} =  -\log f( \exp([p_i])) =-[q_i] =[q_i^\perp].
\]
So, the map $\theta$  induced by $\iota\circ f$ satisfies $ \theta(p_1)+\theta(p_2)+\theta(p_3) = q_1^\perp+q_2^\perp+q_3^\perp =e$, as the $q_i^\perp$ are pairwise orthogonal in case $(ii)$.
\end{proof}

It follows from Lemma \ref{triangles} that if $\Delta(p_1,p_2,p_3)$ is an orthogonal simplex, then the restriction of $f$ to $\Delta(p_1,p_2,p_3)$ is a Hilbert's metric isometry onto either $\Delta(\theta(p_1),\theta(p_2),\theta(p_3))$  or $\Delta(\theta(p_1)^\perp,\theta(p_2)^\perp,\theta(p_3)^\perp)$.  The Hilbert's metric isometries between simplices  have been characterized, see  \cite{dlH} or \cite{LW}, and yields the following dichotomy, as $f(\ol{e})=\ol{e}$.
The isometry  $f$ maps $\Delta(p_1,p_2,p_3)$ onto $\Delta(\theta(p_1),\theta(p_2),\theta(p_3))$ in Lemma~\ref{triangles} if and only if the restriction of $f$ to $\Delta(p_1,p_2,p_3)$ is of the form,
\[
\ol{\lambda_1 p_1+\lambda_2 p_2+\lambda_3 p_3}\,\mapsto\, \ol{\lambda_1 \theta(p_1)+\lambda_2 \theta(p_2)+\lambda_3 \theta(p_3)},
\]
which is equivalent to saying that the restriction of $f$ to  $\Delta(p_1,p_2,p_3)$  is projectively linear.
On the other hand, the  isometry  $f$ maps $\Delta(p_1,p_2,p_3)$ onto $\Delta(\theta(p_1)^\perp,\theta(p_2)^\perp,\theta(p_3)^\perp)$ in Lemma \ref{triangles} if and only if the restriction of $f$ to $\Delta(p_1,p_2,p_3)$ is of the form,
 \[
\ol{\lambda_1 p_1+\lambda_2 p_2+\lambda_3 p_3}\,\mapsto\, \ol{\lambda_1^{-1} \theta(p_1)+\lambda_2^{-1} \theta(p_2)+\lambda_3^{-1} \theta(p_3)},
\]
which is equivalent to saying that the restriction of $\iota \circ f$ to  $\Delta(p_1,p_2,p_3)$  is projectively linear. The above discussion yields the following corollary.

\begin{corollary}\label{cor:proj_linear}
Let $f\colon \ol{M}_+^\circ\to \ol{N}_+^\circ$ be a bijective Hilbert's metric isometry with $f(\ol{e}) =\ol{e}$ and let $\Delta(p_1,p_2,p_3)$ be an orthogonal simplex in $\ol{M}_+^\circ$. Then either $f$ or $\iota\circ f$ is projectively linear on $\Delta(p_1,p_2,p_3)$, and its induced map $\theta$ preserves the orthogonality of $p_1,p_2$ and $p_3$.
\end{corollary}

Our next proposition states that if two orthogonal simplices have a line in common, then $f$ is projectively linear on one simplex if and only if it projectively linear on the other one. The proof uses, among other things,  the following well known fact. If $a,b\in M_+^\circ$ are such that the line through $a$ and $b$ intersect $\partial  M_+$ in $a'$ and $b'$ such that $a$ is between $b$ and $a'$, $b$ is between $a$ and $b'$, then
\begin{equation}\label{ratios}
M(a/b) = \frac{\|a-b'\|}{\|b-b'\|}\mbox{\quad and \quad } M(b/a) = \frac{\|b-a'\|}{\|a-a'\|}.
\end{equation}
A proof can be found in \cite[Chapter 2]{LNBook}.

\begin{proposition}\label{proj on triangles}
Let $f\colon \ol{M}_+^\circ\to \ol{N}_+^\circ$ be a bijective Hilbert's metric isometry with $f(\ol{e}) =\ol{e}$. Let $\Delta(p_1,p_2,p_3)$ and $\Delta(p_4,p_5,p_6)$ be two distinct orthogonal simplices in $\ol{M}_+^\circ$ such that either $p_3=p_6$ or $p_3=p^\perp_6$, so they share the segment $(\ol{p}_3,\ol{p}_3^\perp)$. Then $f$ is projectively linear on $\Delta(p_1,p_2,p_3)$ if and only if it is projectively linear on $\Delta(p_4,p_5,p_6)$.
\end{proposition}
\begin{proof}
Suppose for the sake of contradiction that $f$ is projectively linear $\Delta(p_1,p_2,p_3)$, but not on  $\Delta(p_4,p_5,p_6)$. Denote the image of $\Delta(p_1,p_2,p_3)$  by $\Delta(q_1,q_2,q_3)$, and  the image of $\Delta(p_4,p_5,p_6)$ by
$\Delta(q_4^\perp,q_5^\perp,q_6^\perp)$ as in Lemma \ref{triangles}.  There are 2 cases to consider: $p_3=p_6$ and $p_3=p_6^\perp$. Let us first assume that $p_3=p_6$.

In that case the orthogonal simplices  $\Delta(p_1,p_2,p_3)$  and $\Delta(p_4,p_5,p_6)$ are configured as in Figure~\ref{pyramid}. We will show that
\begin{equation}\label{eq:pyramid}
\mathrm{aff}(p_1,p_2,p_3,p_4,p_5) \cap M_+ =\mathrm{conv}(p_1,p_2,p_3,p_4,p_5).
\end{equation}
\begin{figure}[H]
\[
\begin{tikzpicture}[x=1cm, y=1cm]
\draw[thick](0,0)--(0.6,4.8);
\draw[thick](0,0)--(-0.5,2);
\draw[thick](-0.5,2)--(0.6,4.8);
\draw[thick](0,0)--(4,1.6);
\draw[thick](4,1.6)--(0.6,4.8);
\draw[dashed, thick](-0.5,2)--(2.2,2.6);
\draw[dashed, thick](0.6,4.8)--(2.2,2.6);
\draw[dashed, thick](2.2,2.6)--(4,1.6);
\draw[dashed](0,0)--(2.2,2.6);
\draw[dashed](-0.5,2)--(4,1.6);
\node[above]at(0.6,4.8){$\ol{p}_3$};
\node[left]at(-0.5,2){$\ol{p}_1$};
\node[right]at(2.1,2.8){$\ol{p}_5$};
\node[below]at(0,0){$\ol{p}_4$};
\node[right]at(4,1.6){$\ol{p}_2$};
\node[right]at(1.4,1.5){$\ol{p}_3^\perp$};

\end{tikzpicture}
\]
\caption{Pyramid}
\label{pyramid}
\end{figure}

However, before we do that we consider the  situation for the orthogonal simplices $\Delta(q_1,q_2,q_3)$ and $\Delta(q_4^\perp,q_5^\perp,q_3^\perp)$, which  are configured as in Figure \ref{3simplex}.
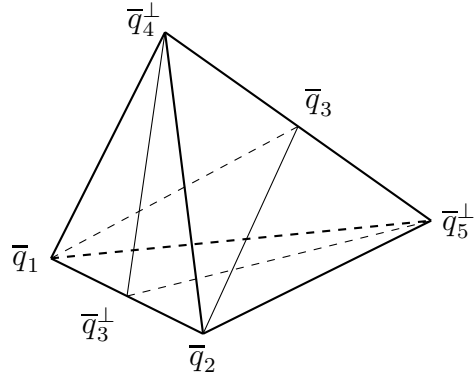
\begin{figure}[h]
\[
\begin{tikzpicture}[x=1cm, y=1cm]
\draw[thick](0,0)--(1.5,3);
\draw[thick](0,0)--(2,-1);
\draw[thick](1.5,3)--(2,-1);
\draw[thick](2,-1)--(5,0.5);
\draw[thick](1.5,3)--(5,0.5);
\draw[dashed, thick](0,0)--(5,0.5);
\node[left] at (0,0){$\ol{q}_1$};
\node[above left]at(1.6,2.8){$\ol{q}_4^\perp$};
\node[below]at(2,-1){$\ol{q}_2$};
\node[right]at(5,0.5){$\ol{q}_5^\perp$};
\draw(1,-0.5)--(1.5,3);
\draw[dashed](1,-0.5)--(5,0.5);
\draw(2,-1)--(3.25,1.75);
\draw[dashed](0,0)--(3.25,1.75);
\node[above right]at(3.20,1.75){$\ol{q}_3$};
\node[below left]at(1,-0.5){$\ol{q}_3^\perp$};
\end{tikzpicture}
\]
\caption{3-simplex}
\label{3simplex}
\end{figure}
Note that as $q_1+q_2 =q_3^\perp$ and $q_4^\perp+q_5^\perp=q_3$ we get that $q_1+q_2+q_4^\perp+q_5^\perp =e$. So, it follows from Lemma \ref{ksimplices} that
\[
\mathrm{aff}(q_1,q_2,q_4^\perp,q_5^\perp)\cap N_+=\mathrm{conv}(q_1,q_2,q_4^\perp,q_5^\perp).
\]

We will now show equality (\ref{eq:pyramid}). Note that $\frac{1}{2}p_2^\perp$, $\frac{1}{2}p_5^\perp$ and $\frac{1}{3}e$ are in $\mathrm{conv}(p_1,p_2,p_3,p_4,p_5)$.
Suppose, for the sake of contradiction, that $\frac{1}{2}(\frac{1}{2}p_2^\perp+\frac{1}{2}p_5^\perp)\not\in \partial M_+$. We know from \cite[Theorem~5.2]{KN} that if we have sequences
\[
b_2(t_n):= (1-t_n){\textstyle\frac{1}{3}}e +t_n {\textstyle\frac{1}{2}}p_2^\perp
\mbox{\quad and\quad }
b_5(s_n):= (1-s_n){\textstyle\frac{1}{3}}e +s_n {\textstyle\frac{1}{2}}p_5^\perp,
\]
with $s_n,t_n\in [0,1)$ such that $t_n\to 1$ and $s_n\to 1$ as $n\to\infty$, then the Gromov product
\[
(b_2(t_n)\mid b_5(s_n))_{e}:=\frac{1}{2}\Big{(}d_H(b_2(t_n),e)+d_H(b_5(s_n),e) - d_H(b_2(t_n),b_5(s_n))\Big{)}
\]
satisfies
\begin{equation}\label{grom1}
\limsup_{n\to \infty}\, (b_2(t_n)\mid b_5(s_n))_{e}<\infty.
\end{equation}
Note that $(\ol{p}_2,\ol{p}_2^\perp)$ and $(\ol{p}_5,\ol{p}_5^\perp)$ are unique geodesics in $(\ol{M}_+^\circ,d_H)$. So, the image of $[\ol{e},\ol{p}_2^\perp)$ under $f$ is the segment $[\ol{e},\ol{q}_2)$, and the image of $[\ol{e},\ol{p}_5^\perp)$  is $[\ol{e},\ol{q}_5^\perp)$. Let us now consider representations of these segments in $\mathrm{conv}(q_1,q_2,q_4^\perp,q_5^\perp)$. It is easy to verify that $\frac{1}{4}e$, $\frac{1}{3}q_2^\perp$ and $\frac{1}{3}q_5$ lie inside $\mathrm{conv}(q_1,q_2,q_4^\perp,q_5^\perp)$. Now for $n\geq 1$ select $a_n$ from the segment $[\frac{1}{4}e,\frac{1}{3}q_2^\perp)$ and $b_n$ from the segment $[\frac{1}{4}e,q_5^\perp)$ such that $a_n\to \frac{1}{3} q_2^\perp$, $b_n\to q_5^\perp$, and the segment $[a_n,b_n]$ is parallel to the segment $[\frac{1}{3}q_2^\perp,q_5^\perp]$.
\begin{figure}[h]
\[
\begin{tikzpicture}[x=1cm, y=1cm]
\draw[thick](0,2)--(5,2);
\draw[dashed, thick](.4,1.2)--(4.61,1.2);
\draw[thick](0,2)--(1,0);
\draw[thick](5,2)--(4,0);
\draw[dashed, thick](5,2)--(1,0);
\draw[dashed, thick](2.5,2)--(2.5,0);

\node[above] at (-0.2,2){$c$};
\node[above] at (2.5,2){$\frac{1}{3}q_2^\perp$};
\node[above] at (5.2,2){$q_5^\perp$};
\node[left] at (2.5,1.5){$a_n$};
\node[left] at (3.5,1.5){$b_n$};
\node[left] at (3.25,.4){$\frac{1}{4}e$};
\node[left] at (1,0){$\frac{1}{3}q_5$};
\node[left] at (.3,1.2){$a'_n$};
\node[right] at (4.7,1.2){$b'_n$};

\node at (2.5,1.2){\tiny{$\bullet$}};
\node at (3.4,1.2){\tiny{$\bullet$}};
\node at (2.5,.75){\tiny{$\bullet$}};
\end{tikzpicture}
\]
\caption{parallel segments}
\label{parallel}
\end{figure}
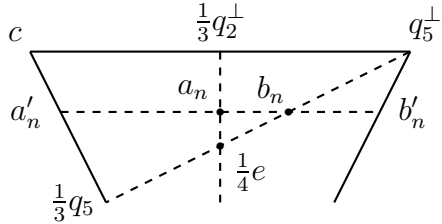

Let $c$, $a'_n$, and $b'_n$ be in the boundary of $\mathrm{conv}(q_1,q_2,q_4^\perp,q_5^\perp)$ as in Figure \ref{parallel}. Then the triangles with vertices $b_n$, $b'_n$ and $q^\perp_5$ are similar for all $n\geq 1$. Hence there exists a constant $C>0$ such that
\[
\frac{\|b_n-b'_n\|}{\|b_n-q_5^\perp\|}=C\mbox{\quad for all }n\geq 1.
\]
Now using (\ref{ratios}) we deduce that
\begin{eqnarray*}
 d_H(b_n, e) -d_H(a_n,b_n) & = &  d_H(b_n,{\textstyle\frac{1}{4}}e) -d_H(a_n,b_n)\\
 & = &  \log\left(\frac{\|b_n-\frac{1}{3}q_5\|}{\|\frac{1}{4}e -\frac{1}{3}q_5\|}
  \frac{\|\frac{1}{4}e-q_5^\perp\|}{\|b_n -q_5^\perp\|}\right) -
  \log\left(\frac{\|a'_n-b_n\|}{\|a'_n -a_n\|}
  \frac{\|a_n -b'_n\|}{\|b_n -b'_n\|}\right)\\
  & \to & C + \log\left(\frac{\|q_5^\perp-\frac{1}{3}q_5\| \|\frac{1}{4}e-q_5^\perp\|}{\|\frac{1}{4}e -\frac{1}{3}q_5\|}
\right) -  \log\left(\frac{\|c-q_5^\perp\| \|\frac{1}{3}q_2^\perp -q_5^\perp\|}{\|c -\frac{1}{3}q_2^\perp\|}\right).
\end{eqnarray*}
Thus, there exists a constant $C'>0$ such that
\[
2(a_n\mid b_n)_e \geq d_H(a_n,e) +C'\mbox{\quad for all }n\geq 1,
\]
which shows that
\[
\limsup_{n\to\infty}\, (a_n\mid b_n)_e=\infty.
\]
As $f^{-1}$ is an isometry and $f(\ol{e})= \ol{e}$, we get that
\[
\limsup_{n\to\infty} \,(f^{-1}(\ol{a}_n)\mid f^{-1}(\ol{b}_n))_{\ol{e}} = \limsup_{n\to\infty}\, (\ol{a}_n\mid \ol{b}_n)_{\ol{e}} = \limsup_{n\to\infty}\, (a_n\mid b_n)_e=\infty.
\]
By construction, however, $f^{-1}(\ol{a}_n) = \ol{b_2(t_n)}$ and $f^{-1}(\ol{b}_n) = \ol{b_5(s_n)}$ for some sequences $(t_n)$ and $(s_n)$ in $[0,1)$ with $t_n,s_n\to 1$, which contradicts (\ref{grom1}).

Thus, $\frac{1}{2} (p_2^\perp+p_5^\perp)\in\partial M_+$ and hence $\mathrm{conv}(p_1,p_3,p_4)\subseteq \partial M_+$. The same argument works for the other  faces containing $p_3$. The square face is also contained in $\partial M_+$, as it contains $\frac{1}{2}p_3^\perp$. This proves \eqref{eq:pyramid}.

Next, we will show that the pre-image of the simplex  $\ol{\mathrm{conv}(q_1,q_2,q_4^\perp,q_5^\perp)}$ lies inside the pyramid  $\ol{\mathrm{conv}(p_1,p_2,p_3,p_4,p_5)}$.
Suppose that $c$ is a point on the segment $(q_2,q_5^\perp)$. The triangle $\mathrm{conv}(c,q_1,q_4^\perp)$ intersects the triangles $\mathrm{conv}(q_1,q_2,q_3)$ and  $\mathrm{conv}(q_3^\perp,q_4^\perp,q_5^\perp)$ in a line segment, say $\gamma_1$ and $\gamma_2$ respectively, see Figure \ref{intersect}. Now suppose that $a\in \mathrm{conv}(c,q_1,q_4^\perp)\cap N_+^\circ$ and let $b$ be the point of intersection of the line segment from $c$ through $a$ with $\mathrm{conv}(q_3^\perp,q_4^\perp,q_5^\perp)$.

\begin{figure}[h]
\[
\begin{tikzpicture}[x=1cm, y=1cm]
\path[fill=lightgray](0,0)to(4,0)to(1.5,3)to(0,0);
\draw[thick](0,0)--(1.5,3);
\draw[thick](0,0)--(2,-1);
\draw[thick](1.5,3)--(2,-1);
\draw[thick](2,-1)--(5,0.5);
\draw[thick](1.5,3)--(5,0.5);
\draw[dashed, thick](0,0)--(5,0.5);
\node[left] at (0,0){$q_1$};
\node[above left]at(1.6,3){$q_4^\perp$};
\node[below]at(2,-1){$q_2$};
\node[right]at(5,0.5){$q_5^\perp$};
\draw[thick](1.5,3)--(4,0);
\draw[dashed, thick](0,0)--(4,0);
\node[below right]at(4,0.1){$c$};
\draw[dashed](0,0)--(3,1.2);
\draw[dashed](1.5,3)--(3,0);
\node[below right]at(4,0.1){$c$};
\node[left]at(.7,1.55){$b$};
\node[above]at(1.4,1.1){$a$};
\draw[dashed](4,0)--(.7,1.4);
\node at(3.2,1.3){$\gamma_1$};
\node at(3,-0.26){$\gamma_2$};
\node at(1.4,1.1){\tiny{$\bullet$}};
\node at(2.05,0.82){\tiny{$\bullet$}};
\node at(2.735,0.545){\tiny{$\bullet$}};
\end{tikzpicture}
\]
\caption{Intersections}
\label{intersect}
\end{figure}
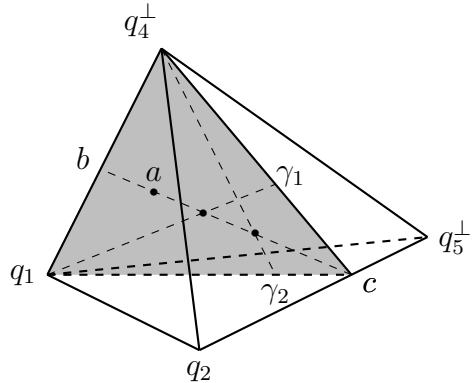

The segment $(\ol{c},\ol{b})$ is  a unique geodesic by Lemma  \ref{ksimplices}. So, its pre-image is projectively a line segment, as $f^{-1}$ is an isometry. Now suppose that  $(c,b)$ intersects $\gamma_1$ and $\gamma_2$ in two distinct points. In that case it follows that the pre-image of $(\ol{c},\ol{b})$  lies  inside  $\ol{\mathrm{conv}(p_1,p_2,p_3,p_4,p_5)}$. The collection of the points $\ol{a}$ for which we obtain such a pre-image forms a dense set of $\ol{\mathrm{conv}(c,q_1,q_4^\perp)}$. So,  by continuity of $f^{-1}$ we conclude that
\[
f^{-1}(\ol{\mathrm{conv}(q_1,q_2,q_4^\perp,q_5^\perp)})\subseteq \ol{\mathrm{conv}(p_1,p_2,p_3,p_4,p_5)}.
\]

It turns out that this situation yields the desired contradiction to prove our assertion in this case.
Let $\rho$ be in the relative interior of $\mathrm{conv}(q_1,q_2,q_5^\perp)$. Then $(\ol{\rho},\ol{q}_4^\perp)$ is a unique geodesic by Lemma~\ref{ksimplices}. Moreover, we have that the segment $(\ol{\rho},\ol{q}_4^\perp)$ is  parallel to $(\ol{q}_4,\ol{q}_4^\perp)$, that is to say
\[
\limsup_{ t\to 0}\,d_H((1-t)q_4^\perp+t\rho,(1-t)q_4^\perp+tq_4) <\infty\mbox{\quad and\quad }
\limsup_{ t\to 1}\,d_H((1-t)q_4^\perp+t\rho,(1-t)q_4^\perp+tq_4) <\infty.
\]
This implies that pre-images of $(\ol{\rho},\ol{q}_4^\perp)$ must also  be parallel  segments. As the pre-image of $ (\ol{q}_4,\ol{q}_4^\perp)$ is $(\ol{p}_4,\ol{p}_4^\perp)$ we find the pre-image of  $(\ol{\rho},\ol{q}_4^\perp)$ is  of the form $(\ol{p}_4,\ol{\sigma})$, with $\sigma$ on the segment $(p_3,p_5)$.  Since  $\rho$ was chosen arbitrarily, this shows that the pre-image of $\ol{\mathrm{conv}(q_1,q_2,q_4^\perp,q_5^\perp)}$ lies  in $\Delta(p_3,p_4,p_5)$, which is absurd. We therefore conclude that $f$ is projectively linear on $\Delta(p_4,p_5,p_6)$ as well.

In case  $p_3=p_6^\perp$ and  $f$ is not projectively linear on $\Delta(p_4,p_5,p_6)$,  then analogously we find that $\mathrm{conv}(p_1,p_2,p_3,p_4,p_5,p_6)$ is the interior of a 3-simplex and $\mathrm{conv}(q_1,q_2,q_3,q_4^\perp,q_5^\perp,q_6^\perp)$ is the interior of a pyramid.  Now applying the same arguments above to $f^{-1}$ yields the desired contradiction, which completes the proof.
\end{proof}

\begin{theorem}\label{t:connecting_simplices}
Let $\Delta(p_1,p_2,p_3)$ and $\Delta(p_4,p_5,p_6)$ be orthogonal simplices in $\ol{M}_+^\circ$. A bijective Hilbert's metric isometry $f\colon\ol{M}_+^\circ\to\ol{N}_+^\circ$ with $f(\ol{e})=\ol{e}$ is projectively linear on $\Delta(p_1,p_2,p_3)$ if and only if it is projectively linear on $\Delta(p_4,p_5,p_6)$.
\end{theorem}
Theorem \ref{t:connecting_simplices} is a simple consequence  from the following  lemma, which uses the following concept. If $p$ and $q$ are nonmaximal nontrivial projections, then by $p \approx q$ we mean that there exists a sequence of nonmaximal projections $p=p_1, \ldots, p_n = q$ such that $p_i \perp p_{i+1}$ and $p_i + p_{i+1} < e$ for $1 \leq i < n$. This defines an equivalence relation on the nonmaximal nontrivial projections in $\P(M)$. 

\begin{lemma}\label{equivalence}
If  $p$  and $q$ are  nonmaximal  nontrivial  projections in a JBW-algebra $M$, then $p\approx q$.
\end{lemma}

If we assume Lemma \ref{equivalence}  for the moment, the proof of Theorem \ref{t:connecting_simplices} goes as follows.
\begin{proof}[Proof of Theorem \ref{t:connecting_simplices}]
By Proposition~\ref{proj on triangles}, if two orthogonal simplices have a projection in common, then $f$ is projectively linear on one of them if and only if it is projectively on the other. So, it suffices to connect any two orthogonal simplices with a chain of orthogonal  simplices each having one projection in common. Note that orthogonal simplices are determined by two nonmaximal nontrivial projections $p_1$ and $p_2$ such that $p_1 \perp p_2$ and $p_1 + p_2 < e$: the third projection is then $(p_1 + p_2)^\perp$. Hence a chain of orthogonal simplices having one projection in common, connecting the projections $p$ and $q$, corresponds to a sequence  of nonmaximal nontrivial projections $p=p_1, \ldots, p_n = q$ such that $p_i \perp p_{i+1}$ and $p_i + p_{i+1} < e$ for $1 \leq i < n$. By Lemma \ref{equivalence} we know that such a sequence always exist, and hence we are done.
\end{proof}

The proof of Lemma \ref{equivalence} is quite technical and will be given in the next section.
However, for particular JB-algebras such as $B(H)_{\mathrm{sa}}$ and Euclidean Jordan algebras,
it is fairly easy to show that  Lemma \ref{equivalence} holds. To do this we make the following basic observation.
\begin{lemma}\label{reduction}
Let $M$ be a JBW-algebra and $p,q \in \P(M)$ be nonmaximal and nontrivial.
\begin{enumerate}[(i)]
\item If $p \perp q$, then $p \approx q$.
\item If $p \leq q$, then $p \approx q$.
\item If $p$ and $q$ operator commute, then $p \approx q$.
\end{enumerate}
\end{lemma}
\begin{proof}
For the first assertion, note that if $q \neq p^\perp$ we are done. Also, if $q = p^\perp$, then by nonmaximality of $q$ and $p$, there exist projections $0<p_0< p$ and $0<q_0<q$, so that  $p \approx q_0 \approx p_0 \approx q$. The second assertion follows from $(i)$, as  $p \approx q^\perp \approx q$. To prove the last one recall that the JBW-algebra generated by $p$ and $q$ is associative by \cite[Proposition 1.47]{AS}, and hence it is isomorphic to $C(K)$ for some compact Hausdorff space $K$. By part $(i)$ we may assume $pq \neq 0$, and then $p \approx pq \approx q$ by part $(ii)$.
\end{proof}
Let us now show that Lemma \ref{equivalence} holds  in case $M=B(H)_{\mathrm{sa}}$.
if $\dim H\leq 2$, then all projections in $\P(M)$ are maximal.
So, assume $\dim H\geq 3$. In that case, any two distinct rank 1 projections $p$ and $q$ are equivalent, because the orthogonal complements of the ranges of $p$ and $q$ have codimension 1, and hence their intersection is nonempty. Let $r$ be the orthogonal projection on the intersection. Note that $r$ is nonmaximal, as the range of $r$ has codimension at least 2. Then $p\perp r$ and $r\perp q$ and hence  $p\approx r\approx q$ by Lemma \ref{reduction}$(i)$. To compete the proof we remark that any nonmaximal projection $p$ with rank at least 2 is equivalent to a rank 1 projection. Simply take $x\in H$ in the range of $p$. Then the orthogonal projection $p_x$ on the span of $x$ satisfies $p_x\leq p$, and hence $p_x\approx p$ by Lemma \ref{reduction}$(ii)$.

We see from Lemma \ref{reduction}$(iii)$ that if the center $Z(M)$ is nontrivial, then any nontrivial projection $z\in Z(M)$ yields $p\approx z\approx q$. Indeed, in this case $z^\perp$ also operator commutes with $p$ and $q$, and we are done if either $z$ or $z^\perp$ is nonmaximal. Suppose that they are both maximal. Then they are also both minimal, and therefore $pz\le z$, forcing $pz\in\{0,z\}$, and $pz^\perp\le z^\perp$, forcing $pz^\perp\in\{0,z^\perp\}$. Combining these identities yields
\[
p=pz+pz^\perp\in\{0,z,z^\perp,e\}
\]
which contradicts the nonmaximality of $p$. So, we may assume that $Z(M)$ is trivial, i.e., $M$ is a factor. Thus, the verify that Lemma \ref{equivalence} holds for Euclidean Jordan algebras, we only need to check the simple ones.

\begin{lemma}\label{symcones}
If $M$ is a simple Euclidean Jordan algebra of rank at least 3 and $p,q \in \P(M)$ are nonmaximal and nontrivial, then $p\approx q$.
\end{lemma}
\begin{proof}
Using the classification of simple Euclidean Jordan algebras we know that $M= H_n(R)$ where $n\geq 3$ and $R=\mathbb{R}$, $\mathbb{C}$ or $\mathbb{H}$, or $M = H_3(\O)$.

By Lemma~\ref{reduction}$(ii)$ we may assume that $p$ and $q$ are primitive. It suffices to show the existence of a nontrivial nonmaximal $z\in\P(M)$ that operator commutes with $p$ and $q$ by the above remarks. We know from \cite[Corollary~IV.2.4]{FK} that there exists $w\in M$ such that $w^2=e$ and $U_w(p) =e_{11}$. Note that $U_w e =w^2 =e$, and hence it is a Jordan isomorphism by Corollary \ref{orderisoms}. So, we may also  assume that $p = e_{11}$. The Jordan algebra generated by $p$ and $q$ is isomorphic to $H_2(\R)$ by \cite[Proposition~1.6]{FK} and the isomorphism in the proof of \cite[Proposition~1.6]{FK} sends $e_{11} \in H_2(\R)$ to $p = e_{11} \in M$.

If $I_2 \in H_2(\R)$ corresponds to a nontrivial projection $z$ under this isomorphism, then $z$ operator commutes with $p$ and $q$ and we are done. We will show that it is impossible that $I_2 \in H_2(\R)$ corresponds to $e \in M$. In that case, the element $s = e_{12}+e_{21} \in H_2(\R)$ is in the Peirce 1/2 eigenspace of $e_{11}$ and satisfies $s^2 = I_2$. However, in $H_n(R)$, elements in the Peirce 1/2 eigenspace of $p$ are of the form
\[
A=\left(\begin{array}{cccc}
0                & a_{12} & \ldots & a_{1n}\\
a^*_{12} & 0          & \ldots & 0\\
\vdots & \vdots & \ddots  &  \vdots \\
a^*_{1n}&  0 & \ldots &  0\\
 \end{array}\right).
 \]
 The diagonal of $A^2$ has entries $A^2_{11} = \sum_{i=2}^n |a_{1i}|^2$ and $A_{ii}^2= |a_{1i}|^2$ for $i=2,\ldots,n$, which is not equal to $e=I_n$ for any choice of $a_{12},\ldots,a_{1n} \in R$, as $n\geq 3$.
\end{proof}

\subsection{Proof of Lemma \ref{equivalence}}
The proof of Lemma \ref{equivalence} requires a number of steps. First note that by Lemma \ref{reduction}$(iii)$, it suffices to find a nontrivial projection $z \in\mathcal{P}(M)$ that operator commutes with both $p$ and $q$. Hence we may assume that
\begin{equation}\label{e:inf}
p \wedge q = p \wedge q^\perp = p^\perp \wedge q = p^\perp \wedge q^\perp = 0.
\end{equation}
Indeed, suppose one of them is nonzero, then it operator commutes with $p$ or $p^\perp$ and $q$ or $q^\perp$, and hence it operator commutes with $p$ and $q$.

The idea of the rest of the proof is to use the theory of von Neumann algebras, and so we would like to view $M$ as the set of selfadjoint elements of a von Neumann algebra. Note that if $M$ is of type $I_2$, then \cite[Theorem~6.1.8]{HO} implies that $M$ is a spin factor $H\oplus\R$. However, in a spin factor all nonzero projections are maximal, so $M$ is not of type $I_2$. As mentioned, the procedure will be divided into several steps. In the case where $M$ is the selfadjoint part of a von Neumann algebra, the proof of this lemma is given in Step 2.

\subparagraph{Step 1:}
We can assume that $M$ is not isomorphic to $H_3(\mathbb{O})$ by Lemma~\ref{symcones}. Then by \cite[Theorem~7.2.7]{HO} we have that $M$ is a $JW$-algebra, that is, it can be represented as a $\sigma$-weakly closed Jordan subalgebra of the selfadjoint operators on a complex Hilbert space. By \cite[Theorem~7.3.3]{HO}, it follows that
\[
M = W^*(M)^\alpha_{sa} = \{ x \in W^*(M): \alpha(x) = x = x^* \}
\]
for some von Neumann algebra $W^*(M)$ and a $*$-anti-automorphism $\alpha$ of $W^*(M)$ of order 2. Now $M$ is a subset of a von Neumann algebra, but the $*$-anti-automorphism $\alpha$ is a problem, which we will eliminate.

Let $R := \{ x \in W^*(M): \alpha(x) = x^* \}$. Then $M = R_{sa}$, and by \cite[Theorem~7.3.2]{HO} we have that $R$ is a $\sigma$-weakly closed real $*$-algebra and $W^*(M) = R \oplus i R$. It follows from \cite[Definition~6.1.1]{Li} that $R$ is a \emph{real $W^*$-algebra}. By \cite[Proposition~6.1.2]{Li}, $R$ is isomorphic to a \emph{real von Neumann algebra}, that is, a $\sigma$-weakly closed $*$-subalgebra of $B(H)$, where $H$ is a \emph{real} Hilbert space. Or equivalently, a $*$-subalgebra of $B(H)$ which has a pre-dual. So, we have succeeded at viewing $M$ as the selfadjoint elements of a von Neumann algebra. Unfortunately, it is a real von Neumann algebra instead of a complex one, which will pose some additional difficulties.

\subparagraph{Step 2:}
Let $N \subseteq R$ be the real von Neumann algebra generated by $p$ and $q$. In the case where $M$ is the selfadjoint part of a von Neumann algebra, the reader can regard $N$ as the von Neumann algebra generated by $p$ and $q$, and $R=M\oplus i M$ here. We denote by  $N'$ the commutant of $N$. That is,
\[
N':=\left\{x\in B(H)\colon xy=yx\ \mbox{for all } y\in N\right\}.
\]
It suffices to find a nontrivial projection $z \in N' \cap R$, because then both $z$ and $z^\perp$ commute with $p$ and $q$, and hence operator commute with $p$ and $q$ by \cite[Proposition 1.49]{AS}. Similarly to the discussion preceding Step 1, we can conclude that either $z$ or $z^\perp$ is nonmaximal. So, we may assume that $N' \cap N$ contains no nontrivial projections. We will now generalize the proof of \cite[Theorem~V.1.41]{Tak}, so that it will also be applicable to the real von Neumann algebra case. From equation \eqref{e:inf}, we obtain that $p^\perp q p$ maps $pH$ injectively onto a dense subspace of $p^\perp H$. Let $uh$ be the polar decomposition of $p^\perp qp$. By \cite[Proposition~4.3.4]{Li} we have that $u,h \in N$. Then $u$ is a partial isometry with initial space $pH$ and final space $p^\perp H$, and so $u^* u = p$ and $uu^* = p^\perp$. We will use this partial isometry $u$ to make a matrix unit $\{ e_{11}, e_{12}, e_{21}, e_{22} \}$. That is, the set of elements $\{e_{11},e_{12},e_{21},e_{22}\}$ satisfies the properties
\[
e_{11}+e_{22}=e,\quad e_{ij}^*=e_{ji},\quad\mbox{and}\quad e_{ij}e_{kl}=\delta_{jk}e_{il}\quad\mbox{for }1\le i,j,k,l\le 2.
\]
Let
$$ e_{11} := p, \quad e_{21} := u, \quad e_{12} := u^*, \quad e_{22} := p^\perp, $$
We will use the following notation. If $M$ is an algebra with projection $p \in M$, then we denote the subalgebra $p M p$ by $M_p$. Furthermore, by $\mathbb{M}_2(M_p)$ we mean the $2\times 2$ matrices whose entries are elements of $M_p$.

\begin{lemma}\label{l:matrix}
If $M$ is a (real) von Neumann algebra with a matrix unit $\{ e_{11}, e_{12}, e_{21}, e_{22} \}$, then $M \cong \mathbb{M}_2( M_{e_{11}})$.
\end{lemma}
\begin{proof}
The reader can easily verify that the map $\phi \colon M \to \mathbb{M}_2(M_{e_{11}})$ given by $\phi(x)_{ij} := e_{1i} x e_{j1}$ is a $*$-homomorphism with inverse $\theta \colon \mathbb{M}_2(M_{e_{11}}) \to M$ defined by $\theta(y_{ij}) := \sum_{i,j=1}^2 e_{i1} y_{ij} e_{1j}$.
\end{proof}

We now apply Lemma~\ref{l:matrix} for $M=N$ and $M=R$, which yields that $N\cong\mathbb{M}_2(N_p)$ and $R\cong \mathbb{M}_2(R_p)$. Moreover, since we used the same matrix unit, the inclusion $N\subseteq R$ corresponds to the natural embedding $\mathbb{M}_2(N_p) \subseteq \mathbb{M}_2(R_p)$. It is straightforward to verify that
\begin{equation}\label{e:comm_eq}
N' \cap R = \left\{ \left(
                                          \begin{array}{cc}
                                            x & 0 \\
                                            0 & x \\
                                          \end{array}
                                        \right)
: x \in N'_p \cap R_p \right\}.
\end{equation}
The projection $p = e_{11}$ is nonmaximal, so there exists a nontrivial projection in $R$ which dominates $p$, and has to be of the form
\[
\left(
                  \begin{array}{cc}
                    p & 0 \\
                    0 & z \\
                  \end{array}
                \right)
\]
for some nontrivial projection $z \in \mathcal{P}(R_p)$.

We claim that it now suffices to show that $N_p$ is a trivial von Neumann algebra. Indeed, in that case $N'_p \cap R_p = R_p$, and so by \eqref{e:comm_eq},
\[
\left(
  \begin{array}{cc}
    z & 0 \\
    0 & z \\
  \end{array}
\right)
\in N' \cap R
\]
is a nontrivial projection, as desired. In the case where $M$ is the selfadjoint part of a von Neumann algebra, we can apply \cite[Theorem~V.1.41(ii)]{Tak} to conclude that $N$ is of type $I_2$, and since $N' \cap N$ contains no nontrivial projections, the spectral theorem implies that $N' \cap N$ is trivial and hence $N$ is a factor. Therefore, we must have $N \cong \mathbb{M}_2(\mathbb{C})$. Since we also have that $N \cong \mathbb{M}_2(N_p)$, it follows that $N_p \cong \mathbb{C}$. In the case where $N\subseteq R$ in a real von Neumann algebra, we have to do some more work to show that $N_p \cong \R$.

\subparagraph{Step 3:}
We will need the following lemma.
\begin{lemma}
$N_p$ is generated by $p$ and $pqp$.
\end{lemma}
\begin{proof}
Taking products of $p$ and $q$ repeatedly yields expressions of the form $ \cdots pqpqpq\cdots$. For $r,s \in \{p,q\}$, let $Q(r,s)$ be the set of such expressions that start with $r$ and end with $s$. It follows that $N$ is the closed linear span of $Q(p,p) \cup Q(p,q) \cup Q(q,p) \cup Q(q,q)$. Hence $N_p$ is the closed linear span of $Q(p,p)$. Since $(pqp)^n = (pq)^{n-1} (pqp)$, it follows that $Q(p,p) = \{p\} \cup \{(pqp)^n: n \geq 1 \}$.
\end{proof}
By the above lemma, $N_p$ is generated by $p$ and $pqp$. Since $p$ is the identity on $N_p$, it is commutative and contains $C_\R(\sigma(pqp))$, the continuous real-valued functions on $\sigma(pqp)$, by the continuous functional calculus for real von Neumann algebras \cite[Proposition~5.1.6(2)]{Li}. Therefore, we have that $N_p \subseteq N'_p$, and so
\[
N \cap N'  \cong \left\{ \left(
                                                  \begin{array}{cc}
                                                    x & 0 \\
                                                    0 & x \\
                                                  \end{array}
                                                \right)
\colon x \in N_p \cap N'_p \right\} = \left\{ \left(
                                                               \begin{array}{cc}
                                                                 x & 0 \\
                                                                 0 & x \\
                                                               \end{array}
                                                             \right)
\colon x \in N_p \right\}.
\]
Since $N \cap N'$ contains no trivial projections, we obtain that $N_p$ contains no trivial projections. However, unlike the case of a von Neumann algebra, a real von Neumann algebra without any nontrivial projections need not be trivial (i.e., $\mathbb{C}$, $\mathbb{H}$). But by \cite[Proposition~4.3.4(3)]{Li}, the linear span of the projections is dense in $(N_{p})_{sa}$, and so $(N_{p})_{sa}$ must be trivial. Since $C_\R(\sigma(pqp))\subseteq (N_{p})_{sa}$, this can only happen if $\sigma(pqp)$ consists of a single element, which implies that $N_p\cong\R$, as desired. This completes the proof of Lemma~\ref{equivalence}.

\subsection{Characterization of Hilbert isometries on JBW-algebras}
Using Theorem \ref{t:connecting_simplices} we can now deduce the desired result.

\begin{corollary}\label{uniform lambda}
If $M$ and $N$ are JBW-algebras and $f\colon \ol{M}_+^\circ \to \ol{N}_+^\circ$ is a bijective Hilbert's metric isometry with $f(\ol{e})=\ol{e}$, then either for  $f$ or for $\iota\circ f$  the induced map $\theta\colon  \P(M)\to \P(N)$ is an orthoisomorphism.
\end{corollary}

\begin{proof}Suppose that $p_1,p_2\in \P(M)$ are orthogonal projections. By Lemma~\ref{lem:T_orth_complements} we may assume that $p_1+p_2<e$. Let $p_3:=(p_1+p_2)^\perp$. After possibly composing $f$ with the inversion $\iota$ we may assume that $f$ is projectively linear on $\Delta(p_1,p_2,p_3)$ and so $\theta$ preserves the orthogonality of $p_1,p_2$ and $p_3$ by Corollary~\ref{cor:proj_linear}. Hence $\theta(p_1)$ and $\theta(p_2)$ are orthogonal. By Theorem~\ref{t:connecting_simplices}, $f$ is projectively linear on all other orthogonal simplices as well, so $\theta$ preserves the orthogonality of all noncomplementary orthogonal projections in $\P(M)$.
Applying the same argument to $f^{-1}$ shows that $\theta^{-1}$ also preserves orthogonality.
\end{proof}

By the proof of \cite[Lemma~1]{Dye}, $\theta$ is an order isomorphism and preserves products of operator commuting projections. Our next goal is to show that $\theta$ extends to a Jordan isomorphism. If $M$ and $N$ are Euclidean Jordan algebras, this can be done with a similar argument as used in \cite{Bo}, see Remark~\ref{r:extend_EJA}. We will now explain how to proceed in the general case of JBW-algebras. The reader only interested in the von Neumann algebra case should follow this argument, but instead of the representations \eqref{e:I_2}, each type $I_2$ von Neumann algebra is isomorphic to $L^\infty(\Omega, \mathbb{M}_2(\mathbb{C}))$.

We can write $M=M_2\oplus \tilde{M}$ and $N=N_2\oplus \tilde{N}$ where $M_2$ and $N_2$ are type $I_2$ direct summands, and $\tilde{M}$ and $\tilde{N}$ are JBW-algebras without type $I_2$ direct summands. See \cite[Theorem~5.1.5, Theorem~5.3.5]{HO}. Suppose $\tilde{p}\in\P(M)$ and $\tilde{q}\in\P(N)$ are the central projections such that $\tilde{p}M=\tilde{M}$ and $\tilde{q}N=\tilde{N}$. Since $\theta$ is an order isomorphism, the restriction $\theta|_{\P(\tilde{M})}\colon\P(\tilde{M})\to \P(\theta(\tilde{p})N)$ is an orthoisomorphism. As $\tilde{M}$ has no type $I_2$ direct summand, we can use the following result.

\begin{theorem}[Bunce, Wright]\label{thm:no_type_2_extends} Let $M$ and $N$ be JBW-algebras such that $M$ has no type $I_2$ direct summand. If $\theta\colon\P(M)\to\P(N)$ is an orthoisomorphism, then $\theta$ extends to a Jordan isomorphism $J\colon M\to N$.
\end{theorem}

\begin{proof} The theorem is exactly \cite[Corollary~2]{BW} but for JBW-algebras instead of JW-algebras. This corollary follows from \cite[Proposition~p.\,91]{BW}, and the crucial ingredient here is that any quantum measure on the projection lattice of a JW-algebra extends to a state. But this statement is also true for JBW-algebras by \cite[Theorem~2.1]{BW2}.
\end{proof}

So $\theta|_{\P(\tilde{M})}$ extends to a Jordan isomorphism $\tilde{J}\colon \tilde{p}M\to \theta(\tilde{p})N$. Moreover, $\theta(\tilde{p})=\tilde{q}$. Indeed, the image of $\tilde{p}M$ under $\tilde{J}$ in $N$ contains no type $I_2$ direct summand, hence $\tilde{J}(\tilde{p}M)\subseteq \tilde{q}N$. This implies that $\theta(\tilde{p})\le \tilde{q}$. Applying the same argument to  $\theta^{-1}$ shows that $\theta^{-1}(\tilde{q})\le \tilde{p}$, so $\tilde{p}=\tilde{q}$.

Our next goal is to show that the orthoisomorphism $\theta|_{\P(M_2)}\colon \P(M_2)\to\P(N_2)$ extends to a Jordan isomorphism as well. By \cite[Theorem~2]{S} we can represent
\begin{equation}\label{e:I_2}
M_2\cong \bigoplus_k L^\infty(\Omega_k,V_k)\quad\mbox{and}\quad
N_2\cong \bigoplus_l L^\infty(\Xi_l, V_l)
\end{equation}
where $k,l$ are cardinals, $\Omega_k,\Xi_l$ are measure spaces, $V_i=H_i\oplus\R$ are spin factors with $\dim H_i=i$. We denote the unit in each $V_k$ by $u$. Let $\Omega := \bigsqcup_k \Omega_k$ be the disjoint union of the $\Omega_k$'s. By identifying $f\in L^\infty(\Omega)$ with $\omega\mapsto f(\omega) u$, we can view $L^\infty(\Omega)$ as lying inside $M_2$. It follows that $Z(M_2)=L^\infty(\Omega)$ and if $p := \mathbf{1}_A \in Z(M_2)$, then $Z(pM_2) = L^\infty(A)$. Since $\theta$ preserves operator commutativity, it preserves the center, and it is straightforward to see that $\theta|_{\P(Z(M_2))} \colon \P(Z(M_2)) \to \P(Z(N_2))$ extends to a Jordan isomorphism $T \colon Z(M_2) \to Z(N_2)$.

Let $a\in M_2$. For almost all $\omega\in\Omega$ the element $a(\omega)$ has rank 1 or rank 2, so modulo null sets we can write $\Omega $ as $\Omega=\Omega^1 \sqcup \Omega^2$ where
\[
\Omega^i:=\left\{\omega\in\Omega\colon \#\sigma(a(\omega))=i\right\}.
\]
If we write $q_i:=\mathbf{1}_{\Sigma_i}$ for $i=1,2$, then there exist unique $\alpha\in Z(q_1 M_2)$, $\beta,\gamma\in Z(q_2 M_2)$, and $0\neq p\in\P(q_2M_2)$ with $p(\omega)$ of rank 1 a.e.\ such that
\[
a(\omega):=
\begin{cases}
\alpha(\omega)u & \mbox{ if $\omega \in \Omega^1$}\\
\beta(\omega)p(\omega)+\gamma(\omega)p(\omega)^\perp& \mbox{ if $\omega \in \Omega^2$}\end{cases}
\]
which yields $a=\alpha+\beta p+\gamma p^\perp$ as a unique representation. Define $J_2\colon M_2\to N_2$ by
\[
J_2(a):=T\alpha +T\beta \theta(p)+T\gamma \theta(p)^\perp.
\]

\begin{lemma}\label{lem:nontrivial projections}
$p\in\P(M_2)$ is a.e.\ rank 1 if and only if $qp\neq 0$ and $qp^\perp\neq 0$ for all nonzero central projections $q\in\P(M_2)$.
\end{lemma}

\begin{proof}
Let $A\subseteq \Omega$ be measurable and suppose that $p(\omega)=0$ a.e.\ on $A$. Then $\mathbf{1}_A\in \P(M_2)$ is a central projection and $\mathbf{1}_Ap=0$. Similarly, if $B\subseteq\Omega$ is a measurable set such that $p(\omega)=u$ a.e.\ on $B$, then $\mathbf{1}_Bp^\perp=0$.

Conversely, if $p\in\P(M_2)$ is a.e.\ rank 1, then neither $\mathbf{1}_Ap=0$ nor $\mathbf{1}_Ap^\perp=0$ for all nonzero measurable $A\subseteq\Omega$, which are precisely the nonzero central projections of $\P(M_2)$.
\end{proof}

Since $\theta$ preserves central projections and orthogonality, it maps a.e.\ rank 1 projections to a.e.\ rank 1 projections. Now $a \in \P(M_2)$ if and only if $\alpha,\beta,\gamma\in\P(Z(M_2))$, and in this case, since $T$ extends $\theta|_{\P(Z(M_2))}$,
\[
J_2(a)=T\alpha+T\beta \theta(p)+T\gamma \theta(p)^\perp=\theta(\alpha)+\theta(\beta)\theta(p)+\theta(\gamma)\theta(p)^\perp=\theta(\alpha)+\theta(\beta p)+\theta(\gamma p^\perp)=\theta(a)
\]
as $\theta$ preserves products of operator commuting projections. Therefore $J_2(a)=\theta(a)$ and so $J_2$ extends $\theta$.

For $\mu\in\R$ and the unit $e_2\in M_2$ we have that $J_{2}(a+\mu e_2)=J_2(a)+\mu e_2$, so $J_2$ induces the quotient map $\overline{J}_2\colon[M_2]\to[N_2]$ defined by $\overline{J}_2([a]):=[J_2 a]$. We claim that $\overline{J}_2$ coincides with $S$ on $[M_2]$. To that end, let $a \in M_2$ be such that $a=\alpha+\beta p+\gamma p^\perp$ where $\alpha=\sum_i\alpha_i\mathbf{1}_{A_i},\beta=\sum_j\beta_j\mathbf{1}_{B_j}$, and
$\gamma=\sum_k\gamma_k\mathbf{1}_{C_k}$ are step functions. Since $\theta$ preserves products of operator commuting projections and the fact that $T$ maps step functions to step functions,

\begin{align*}
\overline{J}_2([a])&=[J_2(a)]=[T\alpha+T\beta \theta(p)+T\gamma \theta(p)^\perp]\\&=\sum_i\alpha_i[\theta(\mathbf{1}_{A_i})]+\sum_j\beta_j[\theta(\mathbf{1}_{B_j} p)]+\sum_k\gamma_k[\theta(\mathbf{1}_{C_k} p^\perp)]\\&=\sum_i\alpha_iS\mathbf{1}_{A_i}+\sum_j\beta_jS\mathbf{1}_{B_j} p+\sum_k\gamma_kS\mathbf{1}_{C_k} p^\perp\\&=S[a].
\end{align*}
Now, for general $a=\alpha+\beta p+\gamma p^\perp\in M_2$ let $\alpha'$, $\beta'$, and $\gamma'$ be approximating step functions for $\alpha$, $\beta$, and $\gamma$. If we put $b:=\alpha'+\beta' p+\gamma' p^\perp$, then

\begin{align*}
\|a-b\|&\le\|\alpha-\alpha'\|+\|\beta-\beta'\|+
\|\gamma-\gamma'\|
\end{align*}
and
\begin{align*}
\|J_2(a)-J_2(b)\|&\le\|\alpha-\alpha'\|+\|\beta-\beta'\|+
\|\gamma-\gamma'\|
\end{align*}
as $T$ is an isometry, so both norms can be made arbitrarily small. This implies that

\begin{align*}
\|\overline{J}_2([a])-S[a]\|_v&\le\|\overline{J}_2([a])-\overline{J}_2([b])\|_v
+\|\overline{J}_2([b])-S[b]\|_v+\|S[b]-S[a]\|_v\\&
=\|\overline{J}_2([a])-\overline{J}_2([b])\|_v+\|S([b]-[a])\|_v\\&
\le\|[J_2(a)-J_2(b)]\|_v+\|[b-a]\|_v\\&\le2\|J_2(a)-J_2(b)\|+2\|b-a\|
\end{align*}
can be made arbitrarily small, and we conclude that $\overline{J}_2=S$ on $[M_2]$.

Having this, we will now proceed to show that $J_2$ is linear. Let $\Xi := \bigsqcup_l \Xi_l$ be the disjoint union of the $\Xi_l$'s, and let $\phi$ be a state on $Z(N_2)=L^\infty(\Xi)$. Then $T^*\phi$ is a state on $Z(M_2)= L^\infty(\Omega)$, and define the functionals $\mathrm{tr}\otimes T^*\phi\in M_2^*$ and $\mathrm{tr}\otimes\phi\in N_2^*$ by

\[
(\mathrm{tr}\otimes T^*\phi)(a):=T^*\phi(\omega\mapsto\mathrm{tr}(a(\omega)))\quad\mbox{and}\quad
(\mathrm{tr}\otimes \phi)(b):=\phi(\xi\mapsto\mathrm{tr}(b(\xi))).
\]
Put $M_0:=\ker\mathrm{tr}\otimes T^*\phi$ and $N_0:=\ker\mathrm{tr}\otimes \phi$. Since $e_2\notin M_0$ and $e_2\notin N_0$, the corresponding quotient maps $\pi_1\colon M_0\to [M_2]$ and $\pi_2\colon N_0\to[N_2]$ are linear isomorphisms. Furthermore, we have that $J_2(M_0)\subseteq N_0$. Indeed, if $x\in M_2$, then since $\theta(p)$ is a.e.\ rank 1,
\[
(\mathrm{tr}\otimes \phi)(J_2(a))=(\mathrm{tr}\otimes \phi)(T\alpha+T\beta \theta(p)+T\gamma \theta(p)^\perp)=\phi(2T\alpha+T\beta+T\gamma).
\]
Therefore, for $a\in M_0$ it follows that

\begin{align*}
(\mathrm{tr}\otimes \phi)(J_2(a))&=\phi(2T\alpha+T\beta+T\gamma)=
\phi(T(2\alpha+\beta+\gamma))\\&=T^*\phi(2\alpha+\beta+\gamma)=(\mathrm{tr}\otimes T^*\phi)(a)\\&=0.
\end{align*}
Now, if $a\in M_0$, then $J_2(a)\in N_0$ which shows the last equality of the equation
\begin{equation}\label{e:linear_J}
\pi_2^{-1}\circ\overline{J}_2\circ \pi_1(a)=\pi_2^{-1}\overline{J}_2[a]=\pi_2^{-1}[J_2(a)]=J_2(a),
\end{equation}
hence $J_2|_{M_0}$ is linear. As $M_2=M_0\oplus\R e_2$ and $N_2=N_0\oplus\R e_2$, and we have $J_2(a+\mu e_2)=J_2(a)+\mu e_2$ for all $\mu\in\R$, it follows that $J_2=J_2|_{M_0}\oplus\mathrm{Id}_{\R e_2}$ is linear.

Moreover, we have

\begin{align*}
\|a\|&=\esssup_{\omega\in\Omega}\|a(\omega)\|=\max\{\|\alpha\|_\infty,
\|\beta\|_\infty,\|\gamma\|_\infty\}\\&
=\max\{\|T\alpha\|_\infty,
\|T\beta\|_\infty,\|T\gamma\|_\infty\}\\&=
\esssup_{\xi\in\Xi}\|J_2(a)(\xi)\|\\&=\|J_2(a)\|,
\end{align*}
so $J_2$ is an isometry and therefore a Jordan isomorphism by Corollary \ref{orderisoms} that extends $\theta|_{\P(M_2)}$. The above discussion yields

\begin{corollary}\label{cor:T_extends_J}
If $f\colon\ol{M}_+^\circ\to\ol{N}_+^\circ$ is a bijective Hilbert's metric isometry with $f(\ol{e})=\ol{e}$ such that its induced map $\theta\colon\P(M)\to\P(N)$ is an orthoisomorphism, then $\theta$ extends to a Jordan isomorphism $J\colon M\to N$.
\end{corollary}

We will now show that the quotient map induced by the Jordan isomorphism $J$ above coincides with $S$.

\begin{lemma}\label{lem:J_is_S}
Let $J \colon M \to N$ be a Jordan isomorphism that extends $\theta$. Then $J$ induces the quotient map $\ol{J} \colon [M] \to [N]$ defined by $\ol{J}([a]) := [J(a)]$, which satisfies $\ol{J} = S$.
\end{lemma}

\begin{proof}
Let $b = \sum_{i=1}^n \lambda_i p_i$, where $\lambda_1, \ldots, \lambda_n \in \R$ and $p_1, \ldots, p_n \in \P(M)$ are orthogonal projections. Then
\begin{equation}\label{e:j=s}
\ol{J}[b] = [Jb] = \left[\sum_{i=1}^n\lambda_i \theta(p_i)\right]=\sum_{i=1}^n\lambda_i [\theta(p_i)]=\sum_{i=1}^n\lambda_i S[p_i]= S[b].
\end{equation}
Now let $a \in M$ and $\eps > 0$. By the spectral theorem, let $b$ be as above such that $\norm{a-b} < \eps$. Then $\norm{Ja - Jb} < \eps$, and since $S$ is a $\norm{\cdot}_v$-isometry and $\norm{\cdot}_v \leq 2 \norm{\cdot}$,
\begin{align*}
\|\ol{J}[a] - S[a]\|_v &\le\| \ol{J}[a] - \ol{J}[b] \|_v + \|\ol{J}[b] - \ol{S}[b] \|_v + \| S[b] - S[a] \|_v \\
&= \norm{ [Ja - Jb]}_v + \norm{[b-a]}_v \\
&\leq 2 \norm{Ja - Jb} + 2 \norm{b-a} \\
&< 4 \eps.
\end{align*}
Hence $\ol{J}[a] = S[a]$ for all $[a]\in [M]$.
\end{proof}

\begin{remark}\label{r:extend_EJA}
If $M$ and $N$ are Euclidean Jordan algebras and $\theta \colon \P(M) \to \P(N)$ is an orthoisomorphism, then an easier argument shows that $\theta$ extends to a Jordan isomorphism. Indeed, every $a \in M$ has a unique spectral decomposition $a = \lambda_1 p_1 + \ldots + \lambda_n p_n$, and so we can define $J(a) := \lambda_1 \theta(p_1) + \ldots \lambda_n \theta(p_n)$. Then $J(a + \mu e) = J(a) + \mu e$, so $J$ induces a map $\ol{J} \colon [M] \to [N]$ by $\ol{J}([a]) := [J(a)]$. By \eqref{e:j=s}, $\ol{J} = S$ is linear. Let $M_0$ and $N_0$ be the kernels of the traces in $M$ and $N$ respectively, then $[M] \cong M_0$ and $[N] \cong N_0$. It is clear from the definition of $J$ that it maps $M_0$ into $N_0$, and so \eqref{e:linear_J} implies that $\ol{J} \cong J|_{M_0}$ is linear, thus $J = J|_{M_0} \oplus \mathrm{Id}_{\R e}$ is linear. Since the spectrum and hence the norm is preserved, $J$ is a Jordan isomorphism by Corollary~\ref{orderisoms}.
\end{remark}

We can now prove the following characterization of the Hilbert's metric isometries on cones in JBW-algebras.

\begin{theorem}\label{t:hilbertisoms} If $M$ and $N$ are JBW-algebras, then $f \colon \ol{M}_+^\circ \to \ol{N}_+^\circ$ is a bijective Hilbert's metric isometry if and only if
\begin{equation}\label{hilbisom}
 f(\overline{a}) = \ol{U_b J(a^\epsilon)} \mbox{\quad for all }\ol{a}\in \ol{M}_+^\circ,
 \end{equation}
where $\epsilon\in\{-1,1\}$, $b\in N_+^\circ$, and $J\colon M\to N$ is a Jordan isomorphism. In this case $b\in f(\ol{e})^{\frac{1}{2}}$.
\end{theorem}

\begin{proof}
Let $ f \colon \ol{M}_+^\circ \to \ol{N}_+^\circ$ be a bijective Hilbert's metric isometry. Then we can define a new bijective isometry $g \colon \ol{M}_+^\circ \to \ol{N}_+^\circ$ by
\[
g(\ol{a}) = U_{f(\ol{e})^{-\frac{1}{2}}} f(\ol{a})\mbox{\quad for all }\ol{a}\in \ol{M}_+^\circ.
\]
Note that $g(\ol{e})=\ol{e}$ and hence it follows from Corollary~\ref{uniform lambda} that either $g$ or $\iota\circ g$ has the property that the induced map $\theta\colon \P(M)\to\P(N)$ is an orthoisomorphism. Let $h\in\{g,\iota\circ g\}$ be the map with this property and $J$ be the Jordan isomorphism from Corollary~\ref{cor:T_extends_J}. Note that $J$ induces a map from $\ol{M}^\circ_+$ to $\ol{N}_+^\circ$. Let $a \in M_+^\circ$, then $a = \exp(c)$ for some $c \in M$, and so by Lemma~\ref{lem:J_is_S},

\begin{eqnarray*}
h(\ol{a})  =  \exp(S\log(\ol{\exp (c)})) = \exp(\ol{J}[c]) = \exp([Jc]) = \ol{\exp(Jc)} = \ol{J(\exp(c))} = \ol{J a} = J\ol{a}.
\end{eqnarray*}
Thus, $h$ coincides with $J$ on $\ol{M}_+^\circ$. Since $h\in\{g,\iota\circ g\}$, for either $\epsilon =1$ or
$\epsilon =-1$ we have that
\[
(U_{f(\ol{e})^{-\frac{1}{2}}} f(\ol{a}))^{\epsilon} = J \ol{a}\mbox{\quad for al $\ol{a}\in \ol{M}_+^\circ$},
\]
hence
\[
f(\ol{a}) = U_{f(\ol{e})^{\frac{1}{2}}} (J\ol{a})^{\epsilon} =U_{f(\ol{e})^{\frac{1}{2}}} (\ol{Ja})^{\epsilon}= U_{f(\ol{e})^{\frac{1}{2}}} \ol{J(a^{\epsilon})}=\ol{U_bJ(a^\eps)}
\]
for some $b\in f(\ol{e})^{\frac{1}{2}}$. To complete the proof note that any map  of the form (\ref{hilbisom}) is a bijective Hilbert's metric isometry.
\end{proof}
Theorem \ref{t:hilbertisoms} has the following direct consequence.

\begin{corollary}
Let $M$ and $N$ be JBW-algebras. The metric spaces $(\ol{M}_+^\circ,d_H)$ and $(\ol{N}_+^\circ,d_H)$ are isometric if and only if $M$ and $N$ are Jordan isomorphic.
\end{corollary}

Next, we will describe the isometry group $\mathrm{Isom}(\ol{M}_+^\circ)$ consisting of all bijective Hilbert's metric isometries on $\ol{M}_+^\circ$.  Consider the subgroup $\mathrm{Proj}(M_+)$ of projectivities consisting of maps $\tau\colon \ol{M}_+^\circ\to\ol{M}_+^\circ$ of the form $\tau(\ol{a}) =\ol{Ta}$, where $T\in\mathrm{Aut}(M_+)$.  Note that by Proposition \ref{p:order_isomorphism} elements $\tau$ in $\mathrm{Proj}(M_+)$  can be written as
$\tau(\ol{a}) = U_{\ol{b}} \ol{Ja}$ with $b\in M_+^\circ$ and $J$ a Jordan isomorphism.  So,
\[
(\iota\circ \tau\circ\iota )(\ol{a}) =( U_{\ol{b}} \ol{Ja^{-1}})^{-1}= U_{\ol{b}^{-1}} \ol{(Ja^{-1})^{-1}} = U_{\ol{b}^{-1}}\ol{Ja},
\]
which shows that $\iota\circ \tau\circ \iota\in \mathrm{Proj}(M_+)$, and hence $\mathrm{Proj}(M_+)$ is a normal subgroup of $\mathrm{Isom}(\ol{M}_+^\circ)$. Moreover, the group $C_2$ of order 2 generated by $\iota$ has trivial intersection with $\mathrm{Proj}(M_+)$ if $\ol{M}_+^\circ$ contains an orthogonal simplex. On the other hand, if $\ol{M}_+^\circ$ does not contain an orthogonal simplex, then $\iota$ belongs to $\mathrm{Proj}(M_+)$. Indeed, if $M$ contains no nontrivial projections, then $M=\R$ and $\iota$ is clearly projectively linear here. If $M$ contains a nontrivial projection, then it is minimal and maximal. So, if $p\in M$ is a nontrivial central projection, then $M=M_p\oplus M_{p^\perp}$. Since both $M_p$ and $M_{p^\perp}$ are JBW-algebras which contain no nontrivial projections, we conclude that $M_p\cong M_{p^\perp}\cong\R$ and $M\cong\R^2$. On $(\R^2_+)^\circ$ the inversion map satisfies $\iota(x,y)=(x^{-1},y^{-1})=(xy)^{-1}(y,x)$,
which belongs to $\mathrm{Proj}(M_+)$. Finally, suppose that all nontrivial projections in $M$ are not central. Then $M$ is a factor, and for any nontrivial projection $p$, it follows that $M_p\cong\R$ by the minimality of $p$. This means that all nontrivial projections in $M$ are abelian and their maximality implies that they have central cover $e$. Since we can write $e=p+p^\perp$, we find that $M$ is of type $I_2$. By \cite[Theorem 6.1.8]{HO} we have that $M$ is a spin factor, so $M_+$ is strictly convex. For an order unit space with strictly convex cone there always exists a strictly positive state, thus by \cite[Remark~3.5]{LRW1} all bijective Thompson's metric isometries on $M_+^\circ$ are projective linear order isomorphisms. This implies that $\iota\in\mathrm{Proj}(M_+)$. We have shown that if $M$ is a JBW-algebra such that $\ol{M_+^\circ}$ does not contain an orthogonal simplex, then $M_+$ must be a Lorentz cone (i.e., the cone of a spin factor or $\R^2_+$). To summarize we have the following result.

\begin{proposition}\label{isom group} Let $M$ be a JBW-algebra. If $\ol{M}_+^\circ$ contains an orthogonal simplex, then the group of bijective Hilbert's metric isometries $\mathrm{Isom}(\ol{M}_+^\circ,d_H)$ satisfies
\[
\mathrm{Isom}(\ol{M}_+^\circ,d_H)\cong\mathrm{Proj}(M_+)\rtimes C_2.
\]
If $\ol{M}_+^\circ$ does not contain an orthogonal simplex, then $\mathrm{Isom}(\ol{M}_+^\circ,d_H)\cong \mathrm{Proj}(M_+)$. Moreover, we have that $\mathrm{Isom}(\ol{M}_+^\circ,d_H)\cong \mathrm{Proj}(M_+)$ if and only if $M_+$ is a Lorentz cone.
\end{proposition}

We believe that the results in this section could be extended to general JB-algebras. However, our arguments rely in a crucial way on the existence of nontrivial projections, which may not be present in a JB-algebra. It would also be interesting to know whether it is true that if the Hilbert's metric isometry group  of a cone $C$ in a complete order unit space is not equal to the group of projectivities of $C$, then the order unit space is a JB-algebra. To date no counter example to this statement is known.

\end{document}